\documentclass[11pt,reqno,oneside]{amsart}

\usepackage[T1]{fontenc}  
\usepackage{textcomp}     
\usepackage{lmodern}      
\usepackage{graphicx}
\usepackage{amssymb}
\usepackage{epstopdf}
\usepackage[a4paper, total={5.56in, 9in}]{geometry}
\usepackage{array} 
\usepackage{hyperref}
\usepackage{bm}
\hypersetup{hypertex=true,colorlinks=true,linkcolor=blue,anchorcolor=blue,citecolor=blue}
\usepackage{amsmath,amsfonts,amsthm,mathrsfs,amssymb,cite}
\usepackage[usenames]{color}
\usepackage{xcolor}

\newtheorem{theorem}{Theorem}[section]

\theoremstyle{definition}
\newtheorem{definition}{Definition}[section]

\theoremstyle{remark}
\newtheorem{remark}{Remark}[section]
\numberwithin{equation}{section}
\numberwithin{equation}{section}

\theoremstyle{remark}
\usepackage{cite}
\usepackage{enumerate}
\usepackage{glossaries}
\makeglossaries  
\allowdisplaybreaks

\DeclareMathOperator{\curl}{curl}

\def\dif{\mathop{}\hphantom{\mskip-\thinmuskip}\mathrm{d}}
\let\daccent\d
\let\d\relax
\newcommand\d{\ifmmode\dif\else\expandafter\daccent\fi}
\allowdisplaybreaks[4]

\title[Spectral structures of EEITEP]{Spectral structures of  elastic-electromagnetic transmission eigenvalue problems}

\author{Huaian Diao}
\address{School of Mathematics and Key Laboratory of Symbolic Computation and Knowledge Engineering of Ministry of Education, Jilin University, Changchun, China}
\email{diao@jlu.edu.cn, hadiao@gmail.com}

\author{Xinyu Ding}
\address{Chengdu Liewu High School of Sichuan Province, Chengdu 610021, China}
\email{xydingmath@163.com}

\author{Yueran Geng}
\address{School of Mathematics, Jilin University, Changchun 130012, China\vspace*{-2mm}}
\address{and\vspace*{-2mm}}
\address{Department of Mathematics, City University of Hong Kong, Kowloon, Hong Kong SAR, China}
\email{gengyr23@mails.jlu.edu.cn, yuerageng2-c@my.cityu.edu.hk}

\author{Hongyu Liu}
\address{Department of Mathematics, City University of Hong Kong, Kowloon, Hong Kong SAR, China}
\email{hongyu.liuip@gmail.com, hongyliu@cityu.edu.hk}

\begin{document}
\maketitle

\begin{abstract}

The time-harmonic elastic-electromagnetic interior transmission eigenvalue problem (EEITEP) arises when an elastic body becomes invisible to an incident electromagnetic wave. This spectral problem is typically non-elliptic and non-self-adjoint, making its analysis delicate. In this paper, we study the discreteness of transmission eigenvalues and the boundary localization of the associated eigenfunctions. For a general bounded Lipschitz domain, we prove that the set of positive transmission eigenvalues, if non-empty, is discrete with $\infty$ as its only possible accumulation point. For a radially symmetric domain, we demonstrate the existence of a sequence of transmission eigenvalues and derive their asymptotic behavior. We rigorously show that the corresponding transmission eigenfunctions exhibit boundary localization in their electromagnetic components, whereas the elastic displacement field remains globally distributed throughout the domain. Finally, we derive lower bounds for the $L^{\infty}$-norms of the electromagnetic gradients normalized by their $L^2$-norms, quantifying their blow-up behavior near the boundary along this sequence. These findings reveal a potential spectral mechanism for developing super-resolution imaging methods in elastic-electromagnetic scattering.

		
\end{abstract}

\section{Introduction} \label{section_1}

\subsection{Mathematical setup and physical relevance}

We initially focus on the mathematical model rather than its physical aspects. Let $\Omega \subset \mathbb{R}^3$ be a bounded Lipschitz domain with a connected complement, and let $\bm{\nu}$ denote the unit outward normal to $\partial \Omega$. Let $ c_0, \rho > 0$, and assume that the Lam\'e constants $\lambda, \mu \in \mathbb{R}$ satisfy the strong convexity condition
\begin{align} \label{eq_intro_strongconvexity}
	\mu > 0, \quad 3\lambda + 2\mu > 0.
\end{align}
We consider the elastic-electromagnetic interior transmission eigenvalue problem (EEITEP):
\begin{align} \label{eq_intro_EHuOmega}
\begin{cases}
	\mathbf{curl}\, \bm{E} - \mathrm{i} k \bm{H} = \bm{0}, \quad \mathbf{curl}\, \bm{H} + \mathrm{i} k \bm{E} = \bm{0} \quad & \text{in } \Omega, \\
	\mathcal{L} \bm{u} + k^2 c_0^2 \rho \bm{u} = \bm{0} \quad & \text{in } \Omega, \\
	\frac{\mathrm{i}}{k} \bm{\nu} \times \bm{H} = T_{\bm{\nu}} \bm{u} \quad & \text{on } \partial\Omega, \\
	\bm{\nu} \times \bm{E} = \bm{\nu} \times \bm{u} \quad & \text{on } \partial\Omega,
\end{cases}
\end{align}
where the differential operator $\mathcal{L}$ is defined as
\begin{align*}
	\mathcal{L}\bm{u} := \mu \Delta \bm{u} + (\lambda+\mu) \nabla \nabla \cdot \bm{u},
\end{align*}
and the surface traction operator $T_{\bm{\nu}}$ is given by
\begin{align*}
	T_{\bm{\nu}}\bm{u} := \lambda(\nabla\cdot\bm{u})\bm{\nu} + 2\mu(\nabla^{\mathrm{s}}\bm{u})\bm{\nu}, \quad \nabla^{\mathrm{s}}\bm{u} := \frac{\nabla\bm{u}+(\nabla\bm{u})^{\top}}{2}.
\end{align*}
If \eqref{eq_intro_EHuOmega} admits a non-trivial solution $(\bm{E}, \bm{H}, \bm{u}) \in H(\mathbf{curl}, \Omega) \times H(\mathbf{curl}, \Omega) \times H^1(\Omega)^3$, then $k$ is called an elastic-electromagnetic transmission eigenvalue, and $(\bm{E}, \bm{H}, \bm{u})$ is called an associated transmission eigenfunction. Here, the space $H(\mathbf{curl}, \Omega)$ is defined as
\begin{equation*}
    H(\mathbf{curl}, \Omega) := \{ \bm{P} \in L^2(\Omega)^3; \mathbf{curl}\, \bm{P} \in L^2(\Omega)^3 \},
\end{equation*}
equipped with the standard norm $\|\bm{P}\|^2_{H(\mathbf{curl}, \Omega)} = \|\bm{P}\|^2_{L^2(\Omega)^3} + \|\mathbf{curl}\, \bm{P}\|^2_{L^2(\Omega)^3}$.

The system \eqref{eq_intro_EHuOmega} arises from the interaction between an electromagnetic field and an elastic body. This interaction model was formulated and analyzed in \cite{cakoni2002mathematical}, where uniqueness, equivalent integral equations, and non-local variational formulations were derived. More precisely, let $\Omega$ denote a homogeneous isotropic elastic body with density $\rho$. Its physical properties are described by the Lam\'e constants $\lambda$ and $\mu$, which satisfy the strong convexity condition \eqref{eq_intro_strongconvexity}. The electromagnetic background medium $\mathbb{R}^3 \setminus \Omega$ is characterized by the electric permittivity $\varepsilon_0 > 0$ and magnetic permeability $\mu_0 > 0$. The wave speed in the electromagnetic background medium is given by
\begin{align*}
	c_0 = \frac{1}{\sqrt{\varepsilon_0\mu_0}},
\end{align*} 
and for a given frequency $\omega \in \mathbb{R}_+$, the corresponding wavenumber is
\begin{equation}\label{eq:dfn k}
k = \omega/c_0.
\end{equation}
Consider the time-harmonic incident electromagnetic waves of the form
\begin{align*} 
	\bm{E}^{\mathrm{in}}(\bm{x}, \bm{d}, \bm{p};k) = -\frac{1}{\mathrm{i} k} \mathbf{curl}\, \mathbf{curl}\, (\bm{p} e^{\mathrm{i}k \bm{x} \cdot \bm{d}}), \quad \bm{H}^{\mathrm{in}}(\bm{x}, \bm{d}, \bm{p};k) = \mathbf{curl}\, (\bm{p} e^{\mathrm{i} k \bm{x} \cdot \bm{d}}), \quad \bm{x} \in \mathbb{R}^3,
\end{align*}
where $\bm{d} \in \mathbb{S}^2$ denotes the propagation direction and $\bm{p} \in \mathbb{R}^3$ is a polarization vector satisfying $\bm{p}\cdot\bm{d}=0$. These incident fields satisfy the Maxwell equations
\begin{align*}
	\mathbf{curl}\, \bm{E}^{\mathrm{in}} - \mathrm{i} k \bm{H}^{\mathrm{in}} = \bm{0}, \quad \mathbf{curl}\, \bm{H}^{\mathrm{in}} + \mathrm{i} k \bm{E}^{\mathrm{in}} = \bm{0} \quad \text{in } \mathbb{R}^3.
\end{align*}
The interaction between $(\bm{E}^{\mathrm{in}}, \bm{H}^{\mathrm{in}})$ and the elastic body $\Omega$ generates a scattered electromagnetic field $(\bm{E}, \bm{H})$ in $\mathbb{R}^3 \setminus \overline{\Omega}$ and an elastic displacement field $\bm{u}$ in $\Omega$. The total electromagnetic field is defined as the superposition of the incident and scattered electromagnetic fields:
\begin{align*}
    (\bm{E}^{\mathrm{tot}}, \bm{H}^{\mathrm{tot}}) := (\bm{E}^{\mathrm{in}}, \bm{H}^{\mathrm{in}}) + (\bm{E}, \bm{H}).
\end{align*}
Across the interface $\partial \Omega$, transmission conditions are imposed according to Voigt's model \cite{maugin2013continuum}. In this model, the interaction between the total electromagnetic field and the elastic displacement takes place only through the boundary $\partial \Omega$. The corresponding time-harmonic direct elastic-electromagnetic scattering problem is then formulated as follows \cite{cakoni2002mathematical}
\begin{align} \label{eq_intro_EE_R3}
	\begin{cases}
		\mathbf{curl}\, \bm{E} - \mathrm{i} k \bm{H} = \bm{0}, \quad \mathbf{curl}\, \bm{H} + \mathrm{i} k \bm{E} = \bm{0} \quad & \text{in } \mathbb{R}^3\setminus\overline{\Omega}, \\
		\mathcal{L} \bm{u} + \omega^2 \rho \bm{u} = \bm{0} \quad & \text{in } \Omega,  \\
		\frac{\mathrm{i}}{k} \bm{\nu} \times (\bm{H} + \bm{H}^{\mathrm{in}}) = T_{\bm{\nu}} \bm{u} \quad & \text{on } \partial\Omega, \\
		\bm{\nu} \times (\bm{E} + \bm{E}^{\mathrm{in}}) = \bm{\nu} \times \bm{u} \quad & \text{on } \partial\Omega.
	\end{cases}
\end{align}
The elastic displacement $\bm{u}$ admits the Helmholtz decomposition
\begin{align*}
	\bm{u} = \bm{u}_{\mathrm{p}} + \bm{u}_{\mathrm{s}},
\end{align*} 
where the compressional (P-wave) and shear (S-wave) components are given by
\begin{align*} 
 \bm{u}_{\mathrm{p}} = - \frac{1}{k_{\mathrm{p}}^2} \nabla \nabla \cdot \bm{u}, \quad \bm{u}_{\mathrm{s}} = \frac{1}{k_{\mathrm{s}}^2} \mathbf{curl}\, \mathbf{curl}\, \bm{u} .
\end{align*}
The corresponding compressional and shear wavenumbers are
\begin{align} \label{eq_intro_kpks}
	k_{\mathrm{p}} = \frac{\omega}{c_{\mathrm{p}}}, \quad k_{\mathrm{s}} = \frac{\omega}{c_{\mathrm{s}}},
\end{align}
where the compressional and shear wave speeds are defined as
\begin{align} \label{def_intro_cscp}
	c_{\mathrm{p}} = \sqrt{\frac{\lambda + 2\mu}{\rho}}, \quad c_{\mathrm{s}} = \sqrt{\frac{\mu}{\rho}}.
\end{align}

To complete the formulation of the direct elastic-electromagnetic scattering problem \eqref{eq_intro_EE_R3}, the Silver-M\"uller radiation condition is introduced:
\begin{align} \label{eq_intro_SMRC}
	\lim_{r \to \infty} \left( \bm{H}(\bm{x}) \times \hat{\bm{x}} - \bm{E}(\bm{x}) \right) = \bm{0}, \quad r := |\bm{x}|,
\end{align}
uniformly in all directions $\hat{\bm{x}} := \bm{x}/|\bm{x}| \in \mathbb{S}^2$. The Silver-M\"uller radiation condition \eqref{eq_intro_SMRC} ensures the outgoing nature of the scattered electromagnetic field $(\bm{E}, \bm{H})$. The well-posedness results for \eqref{eq_intro_EE_R3} under \eqref{eq_intro_SMRC} can be found in \cite{bernardo2010analysis, cakoni2002mathematical, GH10, zhu2022recovering}. Furthermore, the radiation condition \eqref{eq_intro_SMRC} implies that the scattered field $(\bm{E}, \bm{H})$, corresponding to the incident plane wave with propagation direction $\bm{d} \in \mathbb{S}^2$ and polarization $\bm{p} \in \mathbb{R}^3$, admits the asymptotic expansion
\begin{align}
	\bm{E}(\bm{x}, \bm{d}, \bm{p};k) &= \frac{e^{\mathrm{i}kr}}{r} \bm{E}^{\infty}(\hat{\bm{x}}, \bm{d}, \bm{p};k) + O(r^{-2}), \label{eq_Einfty}\\
	\bm{H}(\bm{x}, \bm{d}, \bm{p};k) &= \frac{e^{\mathrm{i}kr}}{r} \bm{H}^{\infty}(\hat{\bm{x}}, \bm{d}, \bm{p};k) + O(r^{-2}), \label{eq_Hinfty}
\end{align}
as $r \to \infty$, where $\bm{E}^{\infty}$ and $\bm{H}^{\infty}$ denote the electric and magnetic far-field patterns respectively.

The scattering problem \eqref{eq_intro_EE_R3} under the radiation condition \eqref{eq_intro_SMRC} gives rise to the forward operator 
\begin{align} \label{eq:IP}
	\mathcal{F}\big((\Omega; \lambda, \mu, \rho), (\bm{E}^{\mathrm{in}}, \bm{H}^{\mathrm{in}})\big) = (\bm{E}^{\infty}, \bm{H}^{\infty}),
\end{align}
which maps the elastic body and the incident field to the corresponding far-field patterns. By the well-posedness of the direct problem, the operator $\mathcal{F}$ is continuous. A related inverse problem is to determine the configuration $(\Omega; \lambda, \mu, \rho)$ from knowledge of the incident field and the measured far-field patterns. This inverse problem arises in various applications, such as geophysical exploration \cite{carcione2007wave} and medical imaging \cite{manduca2001magnetic}. In general, however, this inverse problem is ill-posed.

In this paper, our interest lies in the kernel of $\mathcal{F}$, which is closely related to the concept of invisibility in inverse scattering. Indeed, if the far-field pattern $(\bm{E}^{\infty}, \bm{H}^{\infty})$ vanishes for all observation directions $\hat{\bm{x}} \in \mathbb{S}^2$, then Rellich's lemma \cite{colton2019inverse} implies that the scattered field $(\bm{E}, \bm{H})$ vanishes identically in $\mathbb{R}^3 \setminus \overline{\Omega}$. In this case, the elastic body $\Omega$ is invisible to the incident field $(\bm{E}^{\mathrm{in}}, \bm{H}^{\mathrm{in}})$. Furthermore, if $(\bm{E}, \bm{H}) = \bm{0}$ in $\mathbb{R}^3 \setminus \overline{\Omega}$, it follows from the transmission conditions in \eqref{eq_intro_EE_R3} that the triplet $(\bm{E}^{\mathrm{in}}, \bm{H}^{\mathrm{in}}, \bm{u})$ satisfies the EEITEP \eqref{eq_intro_EHuOmega} in $\Omega$. Thus, the restriction of the incident field $(\bm{E}^{\mathrm{in}}, \bm{H}^{\mathrm{in}})$ to $\Omega$, together with the induced displacement $\bm{u}$, constitutes a transmission eigenfunction, and the corresponding wavenumber $k$ is a transmission eigenvalue. This connection demonstrates that non-scattering is naturally linked to the EEITEP, thereby motivating our rigorous study of this problem.

\subsection{Connection to existing results and summary of main findings}

The interior transmission eigenvalue problem arises naturally in inverse scattering theory and has become an important and active area of research. This problem plays a key role in uniqueness, reconstruction, and invisibility cloaking in inverse scattering \cite{cakoni2022inverse}. Introduced by Kirsch \cite{kirsch1986denseness} and Colton and Monk \cite{colton1988inverse}, it is a non-elliptic and non-self-adjoint eigenvalue problem, and its analysis is therefore particularly delicate.

A substantial literature has been devoted to the properties of transmission eigenvalues, including existence, discreteness, and Weyl-type asymptotics. Existence results have been established systematically for acoustic media \cite{paivarinta2008transmission}, and related results for the electromagnetic and elastic cases can be found in \cite{kirsch2009existence, bellis2013nature}. Discreteness is essential not only for spectral theory but also for numerical reconstruction. In particular, in sampling methods for recovering the support of a scatterer, probing frequencies must be chosen away from transmission eigenvalues, since the reconstruction may otherwise fail \cite{colton2019inverse}. 

For acoustic media, discreteness results can be found in \cite{sylvester2012discreteness}. Corresponding results for other physical models, including electromagnetic, elastic, and acoustic-elastic systems, have also been obtained \cite{chesnel2012interior, bellis2013nature, monk2009inverse}. Furthermore, Weyl-type formulas and other spectral asymptotic properties of transmission eigenvalues have been studied extensively \cite{colton2007interior}.  For a comprehensive treatment of these topics, we refer the reader to the monograph  \cite{cakoni2022inverse}.

Beyond spectral properties, the connection between transmission eigenvalues and non-scattering phenomena is also a central aspect of the theory. Invisibility is characterized by the fact that a scatterer cannot be detected by exterior measurements, that is, the scattered field generated by the incident wave vanishes identically outside the scatterer. In particular, every non-scattering wavenumber is an interior transmission eigenvalue, whereas the converse generally fails to hold. We refer to \cite{cakoni2022inverse} for a comprehensive discussion of this relation. Consequently, transmission eigenfunctions are understood to encode rich geometric information about the scatterer, providing a bridge between spectral theory and inverse scattering.

In recent years, substantial progress has been made in the study of geometric behavior of transmission eigenfunctions. Under suitable local regularity assumptions, it was shown in \cite{blaasten2014corners} that transmission eigenfunctions must vanish at corner points. This result was subsequently extended to other geometric singularities, including corners, edges, and points of high curvature \cite{blaasten2018nonradiating, blaasten2021scattering, cakoni2020corner, diao2023geometrical, diao2025non}, where corresponding vanishing properties were established. More recently, new methods based on free boundary problems were introduced into the study of non-scattering phenomena in \cite{cakoni2023singularities} and \cite{salo2021free}.

The non-scattering results described above are local in nature, since they primarily consider the behavior of transmission eigenfunctions near singular boundary points. By contrast, a global feature of transmission eigenfunctions, namely boundary localization, was identified in \cite{chow2021surface, deng2022new, zhou2025quasi}. More precisely, an infinite sequence of transmission eigenfunctions was constructed whose $L^2$-energy is concentrated in a thin neighborhood of the boundary. Unlike corner vanishing or related local phenomena, boundary localization concerns the global spatial distribution of transmission eigenfunctions throughout the entire domain and thus provides a different perspective on their structure. This phenomenon is also relevant to applications such as super-resolution imaging \cite{tilley2020colour}. At present, however, such results have been obtained mainly for scalar models or relatively simple coupled systems. For more complicated multiphysics vector-valued systems, the corresponding spectral properties and the global geometric behavior of transmission eigenfunctions remain largely unexplored.

In this paper, we consider the EEITEP \eqref{eq_intro_EHuOmega}. Compared with scalar acoustic models, the system \eqref{eq_intro_EHuOmega} is substantially more involved, as it consists of three coupled vector-valued partial differential equations describing the interaction between elastic and electromagnetic fields. Moreover, the spectral parameter $k$ appears not only in the governing equations but also explicitly in the boundary transmission conditions. These features introduce significant challenges to both the spectral analysis of the transmission eigenvalues and the investigation of boundary localization.

We now summarize the main findings of our work. To study the discreteness of the interior transmission eigenvalues for the EEITEP \eqref{eq_intro_EHuOmega} defined on a general bounded Lipschitz domain, we first introduce an associated auxiliary system. When the material parameters of this auxiliary system are complex-valued, its well-posedness is established by variational methods. In particular, we show that positive elastic-electromagnetic transmission eigenvalues do not exist in this regime. The case of real-valued material parameters is then considered. We prove that the set of positive transmission eigenvalues, if non-empty, is discrete, with $\infty$ as its only possible accumulation point. Since positive non-scattering wavenumbers form a subset of the positive transmission eigenvalues of the EEITEP \eqref{eq_intro_EHuOmega}, this result implies that the possible positive non-scattering wavenumbers in the elastic-electromagnetic scattering problem \eqref{eq_intro_EE_R3} and \eqref{eq_intro_SMRC} can only form a discrete set. The proof relies on a decomposition of the associated operator into an invertible part and a compact perturbation, followed by an application of analytic Fredholm theory. To the best of our knowledge, this gives the first discreteness result for elastic-electromagnetic transmission eigenvalues under relatively mild assumptions on the geometry and material parameters.

We next study the EEITEP \eqref{eq_intro_EHuOmega} in a radially symmetric domain and provide a rigorous quantitative analysis of the boundary localization and gradient blow-up behavior of the associated transmission eigenfunctions near the boundary. We first establish the existence of a sequence of transmission eigenvalues of \eqref{eq_intro_EHuOmega} and derive their asymptotic expressions as they tend to infinity. In connection with the invisibility phenomenon in the scattering problem \eqref{eq_intro_EE_R3}, this result identifies and quantitatively characterizes a sequence of non-scattering wavenumbers for a spherical elastic scatterer in the high-wavenumber regime. We then analyze a corresponding family of transmission eigenfunctions $(\bm{E}, \bm{H}, \bm{u})$ given by \eqref{def_bl_En}-\eqref{def_bl_un}. It is shown that the electromagnetic components $(\bm{E}, \bm{H})$ are boundary-localized in the sense of Definition \ref{def:bd}, whereas the elastic component $\bm{u}$ does not display such concentration. In particular, the $L^2$-norm of the electromagnetic fields is concentrated in a thin neighborhood of the boundary, while the elastic part remains globally distributed throughout the domain. To quantify the concentration phenomenon of electromagnetic components $(\bm{E}, \bm{H})$, lower bounds are derived for the $L^{\infty}$-norms of their gradients normalized by their $L^2$-norms. These estimates characterize the boundary blow-up rate of the electromagnetic gradients near the boundary and reveal the highly oscillatory nature of the corresponding transmission eigenfunctions near the boundary. These estimates characterize the blow-up of the electromagnetic gradients near the boundary and reveal the highly oscillatory nature of the corresponding transmission eigenfunctions near the boundary. Motivated by this boundary-localized structure, we further discuss in the final section how the corresponding electromagnetic modes may potentially be exploited for super-resolution imaging of elastic scatterers from electromagnetic far-field measurements.

The remainder of the paper is organized as follows. In Section \ref{section_2}, we establish the well-posedness of an auxiliary problem of EEITEP \eqref{eq_intro_EHuOmega}, and prove the discreteness of transmission eigenvalues for \eqref{eq_intro_EHuOmega}. In Section \ref{section_3}, we investigate the boundary localization properties of the corresponding transmission eigenfunctions, showing that the electromagnetic component is boundary-localized, whereas the elastic component is shown not to exhibit such behavior. In Section \ref{section_4}, we derive lower bounds for the $L^{\infty}$-norms of the electromagnetic gradients normalized by their $L^2$-norms near the boundary, and establish their blow-up behavior as the transmission eigenvalues tend to infinity. Finally, in Section \ref{section_5}, we summarize the main results and discuss a potential super-resolution imaging framework for recovering elastic scatterers from electromagnetic far-field measurements.

\section{Discreteness of transmission eigenvalues} \label{section_2}

In this section, we investigate the discreteness of transmission eigenvalues for the EEITEP \eqref{eq_intro_EHuOmega}. We first introduce an auxiliary system and establish its well-posedness for complex-valued material parameters, which excludes positive transmission eigenvalues in that regime. We then turn to real-valued parameters and show that the set of positive transmission eigenvalues, if non-empty, is discrete, with $+\infty$ as its only possible accumulation point.

Consider the auxiliary PDE system
\begin{align} \label{eq_wp_HuOmega0}
	\begin{cases}
		\mathbf{curl}\, \mathbf{curl}\, \bm{H} - k^2 \bm{H} = \bm{0} \quad&\text{in } \Omega, \\
		\mathcal{L} \bm{u} + k^2 c_0^2 \rho \bm{u}=\bm{0} \quad&\text{in } \Omega,  \\
		\frac{\mathrm{i}}{k} \bm{\nu}\times \bm{H} - T_{\bm{\nu}} \bm{u} = \bm{f} \quad&\text{on } \partial\Omega, \\
		\frac{1}{\mathrm{i}k} \bm{\nu}\times \mathbf{curl}\, \bm{H} + \bm{\nu}\times \bm{u} = \bm{g} \quad&\text{on } \partial\Omega, 
	\end{cases}
\end{align}
where $\bm{f}\in H^{-1/2}\left(\partial\Omega\right)^3$ and $\bm{g}\in H^{-1/2}_{\operatorname{div}}\left(\partial\Omega\right)$. The tangential trace space used below is defined by
\begin{align*}
	H^{-1/2}_{\operatorname{div}}\left(\partial\Omega\right) := \left\{\bm{g}\in H^{-1/2}\left(\partial\Omega\right)^3; \bm{\nu} \cdot \bm{g} = 0 \text{ on } \partial \Omega, \operatorname{div}_{\partial\Omega} \bm{g} \in H^{-1/2}\left(\partial\Omega\right)\right\},
\end{align*}
where $\operatorname{div}_{\partial\Omega}$ and $\mathrm{curl}_{\partial\Omega}$ denote the surface divergence and surface curl on $\partial\Omega$ respectively. In fact, the homogeneous system \eqref{eq_wp_HuOmega0} is equivalent to the EEITEP \eqref{eq_intro_EHuOmega}. When $\bm{f}=\bm{g}=\bm{0}$, the system \eqref{eq_wp_HuOmega0} is obtained from \eqref{eq_intro_EHuOmega} by eliminating $\bm{E}$. Conversely, if $\left(\bm{H},\bm{u}\right)$ solves the homogeneous system \eqref{eq_wp_HuOmega0}, then the triple $\left(\bm{E},\bm{H},\bm{u}\right)$, with
\begin{align} \label{def_wp_EH}
	\bm{E}=\frac{\mathrm{i}}{k} \mathbf{curl}\, \bm{H}
\end{align}
also satisfies \eqref{eq_intro_EHuOmega}. Hence the EEITEP \eqref{eq_intro_EHuOmega} and the homogeneous system \eqref{eq_wp_HuOmega0} have the same transmission eigenvalues, and their eigenfunctions are in one-to-one correspondence through the relation \eqref{def_wp_EH}. For simplicity, the eigenvalues of the homogeneous system \eqref{eq_wp_HuOmega0} will also be referred to as elastic-electromagnetic interior transmission eigenvalues. 

To establish well-posedness of the auxiliary PDE system \eqref{eq_wp_HuOmega0} for complex-valued material parameters in a variational framework, and to prepare for the application of analytic Fredholm theory in proving the discreteness of transmission eigenvalues, we introduce the following function spaces. Define
\begin{align*}
	X_{T,0} := \left\{\bm{H}\in H\left(\mathbf{curl},\Omega\right); \nabla \cdot \bm{H} = 0 \text{ in } \Omega, \bm{\nu} \cdot \bm{H} = 0 \text{ on } \partial \Omega\right\}, 
\end{align*}
equipped with the norm
\begin{align*}
	\left\|\bm{H}\right\|_{X_{T,0}} := \left(\left\|\bm{H}\right\|^2_{L^2\left(\Omega\right)^3} + \left\|\mathbf{curl}\,\bm{H}\right\|^2_{L^2\left(\Omega\right)^3}\right)^{1/2}.
\end{align*}
As shown in \cite{monk2003finite}, there exists $\delta_0\in\left(0\right.,1/2\left.\right]$ such that $X_{T,0}$ is continuously embedded into $H^{1/2+\varepsilon}\left(\Omega\right)^3$ for all $0\leq\varepsilon\leq\delta_0$. Moreover, for any $\bm{H}\in X_{T,0}$, one has
\begin{align*}
	\left\|\bm{H}\right\|_{H^{1/2+\varepsilon}\left(\Omega\right)^3} \leq c \left\|\mathbf{curl}\, \bm{H}\right\|_{L^2\left(\Omega\right)^3},
\end{align*}  
where $c$ is a constant independent of $\bm{H}$. We then define 
\begin{align*}
	\widetilde{H} := X_{T,0}\times H^1\left(\Omega\right)^3,
\end{align*} 
equipped with the norm
\begin{align*}
	\left\|\left(\bm{H},\bm{u}\right)\right\|_{\widetilde{H}} := \left(\left\|\bm{H}\right\|^2_{X_{T,0}} + \left\|\bm{u}\right\|^2_{H^1\left(\Omega\right)^3}\right)^{1/2}.
\end{align*}
It follows from Rellich's embedding theorem that the embedding 
\begin{align*}
	\widetilde{H} \hookrightarrow L^2\left(\Omega\right)^3\times L^2\left(\Omega\right)^3
\end{align*}
is compact.

\subsection{Well-posedness of \texorpdfstring{\eqref{eq_wp_HuOmega0}}{the equation} for complex parameters} 

We begin by establishing the well-posedness of \eqref{eq_wp_HuOmega0} for complex-valued material parameters. More precisely, under the admissibility conditions in Definition \ref{def_wp_lamdamurho}, we prove that the system \eqref{eq_wp_HuOmega0} is well-posed whenever $\rho\notin\mathbb{R}$ or $\mu\notin\mathbb{R}$. In particular, the corresponding homogeneous system \eqref{eq_wp_HuOmega0} admits only the trivial solution. This implies that positive transmission eigenvalues, which correspond to the failure of well-posedness, are excluded in the complex-valued setting. The real-valued case will be treated in the next subsection.

Before stating the well-posedness result, we introduce the admissibility conditions on the complex material parameters used throughout this subsection.

\begin{definition} \label{def_wp_lamdamurho}
The complex parameters $\lambda, \mu$ and $\rho$ are said to be admissible if they satisfy $\operatorname{Re} \left(\mu\right)>0$, $\operatorname{Im} \left(\mu\right)\leq0$, $\operatorname{Re} \left(\rho\right)>0$, $\operatorname{Im} \left(\rho\right)\geq0$, $\operatorname{Re} \left(\lambda+\frac{2}{3}\mu\right)>0$, and $\operatorname{Im} \left(\lambda+\frac{2}{3}\mu\right)\leq0$. In addition, when $\mu\notin\mathbb{R}$, we assume that $\operatorname{Im} \left(\lambda+\frac{2}{3}\mu\right)<0$.
\end{definition}

\begin{remark}
The complex coefficients in Definition \ref{def_wp_lamdamurho} model damping effects in the solid, for instance in viscoelastic media \cite{hsiao2000weak}. In the purely elastic case, these parameters reduce to real values satisfying the strong convexity condition \eqref{eq_intro_strongconvexity}.
\end{remark}

We now establish the well-posedness of \eqref{eq_wp_HuOmega0} for admissible material parameters.

\begin{theorem} \label{thm_wp_3}
Consider \eqref{eq_wp_HuOmega0} with $\bm{f}\in H^{-1/2}\left(\partial\Omega\right)^3$ and $\bm{g}\in H^{-1/2}_{\operatorname{div}}\left(\partial\Omega\right)$. Let $\Omega$ be a bounded Lipschitz domain with a connected complement. Assume that $ \lambda, \mu, \rho\in\mathbb{C}$ are admissible in Definition \ref{def_wp_lamdamurho}, and let $ k,c_0 > 0 $. If $\rho\notin\mathbb{R}$ or $\mu\notin\mathbb{R}$, then there exists a unique solution $\left(\bm{H},\bm{u}\right)\in \widetilde{H}$ to \eqref{eq_wp_HuOmega0} such that
\begin{align*}
	\left\|\bm{H}\right\|_{X_{T,0}} + \left\|\bm{u}\right\|_{H^1\left(\Omega\right)^3} \leq c\left(\left\|\bm{f}\right\|_{H^{-1/2}\left(\partial\Omega\right)^3} + \left\|\bm{g}\right\|_{H^{-1/2}_{\operatorname{div}}\left(\partial\Omega\right)}\right),
\end{align*}
where $c$ is a constant independent of $\bm{f}$ and $\bm{g}$. In particular, the EEITEP \eqref{eq_intro_EHuOmega} admits no positive transmission eigenvalues under the above assumptions.
\end{theorem}

\begin{proof}
Let $k>0$. For any $\bm{J}\in X_{T,0}$, integration by parts yields
\begin{align*} 
	\int_{\Omega} \left(\mathbf{curl}\, \bm{H} \cdot \mathbf{curl}\, \overline{\bm{J}} - k^2 \bm{H} \cdot \overline{\bm{J}}\right) \mathrm{d}\bm{x} - \mathrm{i} k \int_{\partial\Omega} \left(\bm{\nu} \times \bm{u}\right) \cdot \overline{\bm{J}} \mathrm{d}\sigma = - \mathrm{i} k \int_{\partial\Omega} \bm{g} \cdot \overline{\bm{J}} \mathrm{d}\sigma.
\end{align*}
By the first Betti's formula, we obtain
\begin{align*}
	\int_{\Omega} \left(\left(\mathcal{C} : \nabla \bm{u}\right) : \nabla \overline{\bm{t}} - k^2 c_0^2 \rho \bm{u} \cdot \overline{\bm{t}}\right) \mathrm{d}\bm{x} + \frac{\mathrm{i}}{k} \int_{\partial\Omega} \left(\bm{\nu} \times \overline{\bm{t}}\right) \cdot \bm{H} \mathrm{d}\sigma = - \int_{\partial\Omega} \bm{f} \cdot \overline{\bm{t}} \mathrm{d}\sigma,
\end{align*}
for all $\bm{t}\in H^1\left(\Omega\right)^3$, where
\begin{align*}
	\mathcal{C} : \nabla \bm{u} := \lambda (\operatorname{div}\bm{u}) \bm{I} + 2\mu\nabla^{\mathrm{s}}\bm{u}.
\end{align*}
Hence the variational problem associated with \eqref{eq_wp_HuOmega0} is formulated as follows: find $\left(\bm{H},\bm{u}\right)\in \widetilde{H}$ such that
\begin{align} \label{eq_wp_varproblem_1} 
	\mathcal{M}_k\left(\left(\bm{H},\bm{u}\right),\left(\bm{J},\bm{t}\right)\right) = \mathcal{R}_k\left(\bm{J},\bm{t}\right), \quad \forall\left(\bm{J},\bm{t}\right)\in \widetilde{H},
\end{align}
where
\begin{align*} 
	\mathcal{M}_k\left(\left(\bm{H},\bm{u}\right),\left(\bm{J},\bm{t}\right)\right)& :=  \int_{\Omega} \left(\mathbf{curl}\, \bm{H} \cdot \mathbf{curl}\, \overline{\bm{J}} - k^2 \bm{H} \cdot \overline{\bm{J}}\right) \mathrm{d}\bm{x} - \mathrm{i} k \int_{\partial\Omega} \left(\bm{\nu} \times \bm{u}\right) \cdot \overline{\bm{J}} \mathrm{d}\sigma \notag \\
	& + k^2 \int_{\Omega} \left(\left(\mathcal{C} : \nabla \bm{u}\right) : \nabla \overline{\bm{t}} - k^2 c_0^2 \rho \bm{u} \cdot \overline{\bm{t}}\right) \mathrm{d}\bm{x} + \mathrm{i} k \int_{\partial\Omega} \left(\bm{\nu} \times \overline{\bm{t}}\right) \cdot \bm{H} \mathrm{d}\sigma,
\end{align*}
and
\begin{align*}
	\mathcal{R}_k\left(\bm{J},\bm{t}\right) := - \mathrm{i} k \int_{\partial\Omega} \bm{g} \cdot \overline{\bm{J}} \mathrm{d}\sigma - k^2 \int_{\partial\Omega} \bm{f} \cdot \overline{\bm{t}} \mathrm{d}\sigma.
\end{align*}
Since $X_{T,0}$ is continuously embedded into $H^{1/2+\varepsilon}\left(\Omega\right)^3$ for all $0<\varepsilon\leq\delta_0$, the trace theorem implies that the sesquilinear form $\mathcal{M}_k\left(\left(\bm{H},\bm{u}\right),\left(\bm{J},\bm{t}\right)\right)$ is continuous on $\widetilde{H}\times\widetilde{H}$. By the Riesz representation theorem, there exists a unique bounded linear operator $A_k : \widetilde{H}\to\widetilde{H}$ such that
\begin{align*}
	(A_k(\bm{H},\bm{u}),(\bm{J},\bm{t}))_{\widetilde{H}} = \mathcal{M}_k((\bm{H},\bm{u}),(\bm{J},\bm{t})), \quad \forall (\bm{J},\bm{t})\in \widetilde{H}.
\end{align*}
We decompose the sesquilinear form $\mathcal{M}_k$ as
\begin{align*}
	\mathcal{M}_k = \mathcal{M}_{k,1} + \mathcal{M}_{k,2},
\end{align*}
where 
\begin{align*}
	\mathcal{M}_{k,1}\left(\left(\bm{H},\bm{u}\right),\left(\bm{J},\bm{t}\right)\right) := & \int_{\Omega} \left(\mathbf{curl}\, \bm{H} \cdot \mathbf{curl}\, \overline{\bm{J}} + M \bm{H} \cdot \overline{\bm{J}}\right) \mathrm{d}\bm{x} \\
	& + k^2 \int_{\Omega} \left(\left(\mathcal{C} : \nabla \bm{u}\right) : \nabla \overline{\bm{t}} + M \bm{u} \cdot \overline{\bm{t}}\right) \mathrm{d}\bm{x},
\end{align*}
and 
\begin{align*}
	\mathcal{M}_{k,2}\left(\left(\bm{H},\bm{u}\right),\left(\bm{J},\bm{t}\right)\right) := & - \int_{\Omega} \left(M+k^2\right) \bm{H} \cdot \overline{\bm{J}} \mathrm{d}\bm{x} - k^2 \int_{\Omega} \left(M+k^2 c_0^2 \rho\right) \bm{u} \cdot \overline{\bm{t}} \mathrm{d}\bm{x} \\
	& - \mathrm{i} k \int_{\partial\Omega} \left(\bm{\nu} \times \bm{u}\right) \cdot \overline{\bm{J}} \mathrm{d}\sigma + \mathrm{i} k \int_{\partial\Omega} \left(\bm{\nu} \times \overline{\bm{t}}\right) \cdot \bm{H} \mathrm{d}\sigma,
\end{align*}
with $M>0$ sufficiently large. Since both $\mathcal{M}_{k,1}$ and $\mathcal{M}_{k,2}$ are bounded, the Riesz representation theorem yields unique bounded linear operators $A_{k,1},A_{k,2} : \widetilde{H} \to \widetilde{H}$ such that
\begin{align*} 
	\left(A_{k,i}\left(\bm{H},\bm{u}\right),\left(\bm{J},\bm{t}\right)\right)_{\widetilde{H}} = \mathcal{M}_{k,i}((\bm{H},\bm{u}),(\bm{J},\bm{t})), \quad i=1,2,
\end{align*}
for all $\left(\bm{J},\bm{t}\right)\in \widetilde{H}$. 

For $A_{k,1}$, coercivity in the $\bm{H}$-component is immediate. For the $\bm{u}$-part, coercivity on $H^1\left(\Omega\right)^3$ follows from Korn's inequality \cite{mclean2000strongly} under the admissibility conditions in Definition \ref{def_wp_lamdamurho}. Hence $A_{k,1}$ is invertible by the Lax-Milgram theorem.

The operator $A_{k,2}$ is compact. Indeed, by Rellich's embedding theorem and the trace theorem, the embeddings
\begin{align*}
	H^1\left(\Omega\right)^3\hookrightarrow L^2\left(\Omega\right)^3, \quad X_{T,0}\hookrightarrow L^2\left(\Omega\right)^3, \quad H^{1/2}\left(\partial\Omega\right)^3\hookrightarrow L^2\left(\partial\Omega\right)^3, \quad 
\end{align*} 
and
\begin{align*}
	H^{\varepsilon}\left(\partial\Omega\right)^3\hookrightarrow L^2\left(\partial\Omega\right)^3, \quad 0<\varepsilon\leq\delta_0
\end{align*}
are compact. Hence $A_{k,2}$ is compact on $\widetilde{H}$. Since $A_{k,1}$ is invertible, it follows that $A_k=A_{k,1}+A_{k,2}$ is a Fredholm operator of index zero. By the Fredholm alternative, existence follows once uniqueness has been established. Since \eqref{eq_wp_HuOmega0} is linear, it remains to prove that the corresponding homogeneous system admits only the trivial solution. 

Let $(\bm{H}, \bm{u})$ be a solution of the homogeneous problem. Since $\mathcal{M}_k\left(\left(\bm{H},\bm{u}\right),\left(\bm{H},\bm{u}\right)\right)=0$, it follows that
\begin{align*}
	\operatorname{Im}\, \mathcal{M}_k\left(\left(\bm{H},\bm{u}\right),\left(\bm{H},\bm{u}\right)\right) = \operatorname{Im}\, \int_{\Omega} \left(\mathcal{C} : \nabla \bm{u}\right) : \nabla \overline{\bm{u}} \mathrm{d}\bm{x} - \operatorname{Im}\, \int_{\Omega} k^2 c_0^2 \rho \left|\bm{u}\right|^2 \mathrm{d}\bm{x} = 0.
\end{align*}
Note that 
\begin{align*}
	\left(\mathcal{C} : \nabla \bm{u}\right) : \nabla \overline{\bm{u}} = \left(\lambda+\frac{2}{3}\mu\right)\left|\operatorname{div} \bm{u}\right|^2 + 2\mu\left\|\nabla^{\mathrm{s}}\bm{u} - \frac{1}{3}(\operatorname{div}\bm{u}) \bm{I}\right\|_{\mathrm{F}}^2,
\end{align*}
where $\left\|\cdot\right\|_{\mathrm{F}}$ denotes the Frobenius norm. By Definition \ref{def_wp_lamdamurho}, we have
\begin{align} \label{eq_wp_uni_4}
	\operatorname{Im}\, \int_{\Omega} \left(\mathcal{C} : \nabla \bm{u}\right) : \nabla \overline{\bm{u}} \mathrm{d}\bm{x} = \operatorname{Im}\, \int_{\Omega} k^2 c_0^2 \rho \left|\bm{u}\right|^2 \mathrm{d}\bm{x} = 0.
\end{align}
If $\rho\notin\mathbb{R}$, then $\operatorname{Im} \left(\rho\right)>0$, and therefore
\begin{align*}
	\int_{\Omega} \left|\bm{u}\right|^2 \mathrm{d}\bm{x} = 0.
\end{align*}
Hence $\bm{u}=\bm{0}$ in $\Omega$. If $\mu\notin\mathbb{R}$, then $\operatorname{Im} \left(\mu\right)<0$ and $\operatorname{Im} \left(\lambda+\frac{2}{3}\mu\right)<0$. It follows from \eqref{eq_wp_uni_4} that
\begin{align*}
	\operatorname{div}\bm{u}=0, \quad \nabla^{\mathrm{s}}\bm{u} - \frac{1}{3}(\operatorname{div}\bm{u}) \bm{I}=0,
\end{align*}
for $\bm{x}\in\Omega$. Therefore we obtain $\mathcal{C} : \nabla \bm{u}=0$ and $\bm{u}=-\frac{1}{k^2 c_0^2\rho}\nabla\cdot\left(\mathcal{C} : \nabla \bm{u}\right)=\bm{0}$ in $\Omega$.
	
Substituting $\bm{u}=\bm{0}$ into \eqref{eq_wp_HuOmega0}, we obtain
\begin{align*}
\begin{cases}
	\mathbf{curl}\, \mathbf{curl}\, \bm{H} - k^2 \bm{H} = \bm{0} \quad&\text{in } \Omega, \\
	\bm{\nu}\times \bm{H} = \bm{0} \quad&\text{on } \partial\Omega, \\
	\bm{\nu}\times \mathbf{curl}\, \bm{H} = \bm{0} \quad&\text{on } \partial\Omega.
\end{cases}
\end{align*}
By the unique continuation principle for the Maxwell system, it follows that $\bm{H}=\bm{0}$ in $\Omega$. 
	
The proof is complete.
\end{proof}

\subsection{Discreteness of transmission eigenvalues of EEITEP \texorpdfstring{\eqref{eq_intro_EHuOmega}}{the equation}}

We now turn to the case where the material parameters $\lambda$, $\mu$ and $\rho$ are real-valued. Since positive transmission eigenvalues are excluded in the complex-valued setting, it remains to study the discreteness of transmission eigenvalues for \eqref{eq_intro_EHuOmega} with real-valued parameters. The following theorem shows that, if positive transmission eigenvalues associated with \eqref{eq_intro_EHuOmega} exist, then they form a discrete set with $+\infty$ as the only possible accumulation point. 

\begin{theorem} \label{thm_wp_1}
Consider the EEITEP \eqref{eq_intro_EHuOmega}. Let $\Omega$ be a bounded Lipschitz domain with a connected complement. Suppose that $\lambda, \mu\in\mathbb{R}$ satisfy \eqref{eq_intro_strongconvexity}, and let $c_0, \rho\in\mathbb{R}_+$. Then the set of positive transmission eigenvalues associated with \eqref{eq_intro_EHuOmega}, if non-empty, is discrete, with $+\infty$ as its only possible accumulation point.
\end{theorem}

\begin{proof}
Note that any transmission eigenvalue of \eqref{eq_intro_EHuOmega}, together with its associated transmission eigenfunction, satisfies the auxiliary system \eqref{eq_wp_HuOmega0} with $\bm{f}=\bm{g}=\bm{0}$. Hence it suffices to analyze the homogeneous system \eqref{eq_wp_HuOmega0}.
	
To establish the discreteness via the analytic Fredholm theory, we extend $k$ to the complex plane. For each $k\in\mathbb{C}\setminus\left\{0\right\}$, the variational formulation of the homogeneous system \eqref{eq_wp_HuOmega0} is to find $\left(\bm{H},\bm{u}\right)\in \widetilde{H}$ such that
\begin{align} \label{eq_wp_varproblem_2}
	\mathcal{N}_k\left(\left(\bm{H},\bm{u}\right),\left(\bm{J},\bm{t}\right)\right) = 0, \quad \forall \left(\bm{J},\bm{t}\right)\in \widetilde{H},
\end{align}
where
\begin{align*} 
	\mathcal{N}_k\left(\left(\bm{H},\bm{u}\right),\left(\bm{J},\bm{t}\right)\right) := & \int_{\Omega} \left(\mathbf{curl}\, \bm{H} \cdot \mathbf{curl}\, \overline{\bm{J}} - k^2 \bm{H} \cdot \overline{\bm{J}}\right) \mathrm{d}\bm{x} - \mathrm{i} k \int_{\partial\Omega} \left(\bm{\nu} \times \bm{u}\right) \cdot \overline{\bm{J}} \mathrm{d}\sigma \notag \\
	& + \int_{\Omega} \left(\left(\mathcal{C} : \nabla \bm{u}\right) : \nabla \overline{\bm{t}} - k^2 c_0^2 \rho \bm{u} \cdot \overline{\bm{t}}\right) \mathrm{d}\bm{x} + \frac{\mathrm{i}}{k} \int_{\partial\Omega} \left(\bm{\nu} \times \overline{\bm{t}}\right) \cdot \bm{H} \mathrm{d}\sigma.
\end{align*}
It follows from the trace theorem that $\mathcal{N}_k$ is continuous on $\widetilde{H}\times\widetilde{H}$. By the Riesz representation theorem, there exists a unique bounded linear operator $B_k : \widetilde{H}\to\widetilde{H}$ such that
\begin{align*} 
	(B_k(\bm{H},\bm{u}),(\bm{J},\bm{t}))_{\widetilde{H}} = \mathcal{N}_k((\bm{H},\bm{u}),(\bm{J},\bm{t})), \quad \forall (\bm{J},\bm{t})\in \widetilde{H}.
\end{align*}
As in the proof of Theorem \ref{thm_wp_3}, the operator $B_k$ can be written as the sum of a bounded invertible operator and a compact operator. Moreover, the map $k\mapsto B_{k}$ is analytic in $\mathbb{C}\setminus\left\{0\right\}$. Since $B_{k}$ is Fredholm of index zero, it suffices to show that there exists $k\in\mathbb{C}\setminus\left\{0\right\}$ such that $B_k$ is injective. 

To this end, let $k=\mathrm{i}\tau$ with $\tau>0$. For each $(\bm{H}, \bm{u})\in\widetilde{H}$ satisfying \eqref{eq_wp_varproblem_2}, we choose the test function $(\bm{H}, \eta \bm{u})$, where $\eta>0$. Then we have
\begin{align*}
	\mathcal{N}_{\mathrm{i}\tau}\left(\left(\bm{H},\bm{u}\right),\left(\bm{H},\eta\bm{u}\right)\right) = I_1 + I_2 + I_3 + I_4, 
\end{align*}
where
\begin{align*}
	& I_1 := \int_{\Omega} \left(\left|\mathbf{curl}\, \bm{H}\right|^2 + \tau^2 \left|\bm{H}\right|^2 \right) \mathrm{d}\bm{x}, 
	& I_2 & := \tau \int_{\partial\Omega} \left(\bm{\nu} \times \bm{u}\right) \cdot \overline{\bm{H}} \mathrm{d}\sigma, \\
	& I_3 := \eta \int_{\Omega} \left(\left(\mathcal{C} : \nabla \bm{u}\right) : \nabla \overline{\bm{u}} + \tau^2 c_0^2 \rho \left|\bm{u}\right|^2 \right) \mathrm{d}\bm{x}, 
	& I_4 & := - \frac{\eta}{\tau} \int_{\partial\Omega} \left(\bm{\nu} \times \bm{H}\right) \cdot \overline{\bm{u}} \mathrm{d}\sigma.
\end{align*}
By Green's formula, we obtain 
\begin{align*}
	\left|I_2\right| & = \left|\tau\int_{\Omega} \left(\mathbf{curl}\, \bm{u} \cdot \overline{\bm{H}} - \bm{u} \cdot \mathbf{curl}\, \overline{\bm{H}}\right) \mathrm{d}\bm{x}\right| \\
	& \leq \left\|\nabla \bm{u}\right\|^2_{L^2\left(\Omega\right)^{3\times3}}  + \frac{\tau^2}{2} \left\|\bm{H}\right\|^2_{L^2\left(\Omega\right)^3} + \frac{\tau^2}{2} \left\|\bm{u}\right\|^2_{L^2\left(\Omega\right)^3} + \frac{1}{2} \left\|\mathbf{curl}\, \bm{H}\right\|^2_{L^2\left(\Omega\right)^3}.
\end{align*}
Similarly, it follows that
\begin{align*}
	\left|I_4\right| & = \left|\frac{\eta}{\tau} \int_{\Omega} \left(\bm{H} \cdot \mathbf{curl}\, \overline{\bm{u}} - \mathbf{curl}\, \bm{H} \cdot \overline{\bm{u}}\right)   \mathrm{d}\bm{x}\right| \\
	& \leq \frac{\eta}{\tau^2}\left\|\nabla \bm{u}\right\|^2_{L^2\left(\Omega\right)^{3\times3}}  + \frac{\eta}{2} \left\|\bm{H}\right\|^2_{L^2\left(\Omega\right)^3} + \frac{\eta}{2} \left\|\bm{u}\right\|^2_{L^2\left(\Omega\right)^3} + \frac{\eta}{2\tau^2} \left\|\mathbf{curl}\, \bm{H}\right\|^2_{L^2\left(\Omega\right)^3}.
\end{align*}
Using Korn's second inequality, we have
\begin{align*}
	I_3 & = \eta \int_{\Omega} \left(\lambda\left|\operatorname{div} \bm{u}\right|^2 + 2\mu \nabla^{\mathrm{s}}\bm{u} : \nabla^{\mathrm{s}}\overline{\bm{u}} + \tau^2 c_0^2 \rho \left|\bm{u}\right|^2 \right) \mathrm{d}\bm{x}, \\
	& \geq \eta \int_{\Omega} \left(\alpha\nabla^{\mathrm{s}}\bm{u} : \nabla^{\mathrm{s}}\overline{\bm{u}} + \tau^2 c_0^2 \rho \left|\bm{u}\right|^2 \right) \mathrm{d}\bm{x}, \\
	& \geq \eta\alpha c_1 \left\|\nabla \bm{u}\right\|^2_{L^2\left(\Omega\right)^{3\times3}} + \eta\left(\tau^2 c_0^2 \rho - \alpha\right) \left\|\bm{u}\right\|^2_{L^2\left(\Omega\right)^3},
\end{align*}
where $\alpha=\min\left\{2\mu,3\lambda+2\mu\right\}>0$ and $c_1>0$ is a constant independent of $u$. Hence we obtain
\begin{align*}
	\operatorname{Re} \left(\mathcal{N}_{\mathrm{i}\tau}\left(\left(\bm{H},\bm{u}\right),\left(\bm{H},\eta\bm{u}\right)\right)\right) \geq & \frac{\tau^2 - \eta}{2} \left\|\bm{H}\right\|^2_{L^2\left(\Omega\right)^3} + \left(\eta\tau^2 c_0^2\rho - \eta\alpha - \frac{\tau^2+\eta}{2}\right) \left\|\bm{u}\right\|^2_{L^2\left(\Omega\right)^3} \\
	& + \frac{\tau^2-\eta}{2\tau^2} \left\|\mathbf{curl}\, \bm{H}\right\|^2_{L^2\left(\Omega\right)^3} + \left(\eta\alpha c_1 - 1 - \frac{\eta}{\tau^2}\right) \left\|\nabla \bm{u}\right\|^2_{L^2\left(\Omega\right)^{3\times3}}.
\end{align*}
By choosing $\eta=\max\left\{\frac{3}{\alpha c_1}, \frac{\alpha+2}{c_0^2\rho}\right\}$ and $\tau=\sqrt{2\eta}$, we deduce that
\begin{align*}
	\operatorname{Re} \left(\mathcal{N}_{\mathrm{i}\sqrt{2\eta}}\left(\left(\bm{H},\bm{u}\right),\left(\bm{H},\eta\bm{u}\right)\right)\right) \geq c_2 \left\|\left(\bm{H},\bm{u}\right)\right\|^2_{\widetilde{H}},
\end{align*}
where $c_2=\min\left\{\eta/2,1/4\right\}>0$. Therefore $B_{\mathrm{i}\sqrt{2\eta}}$ is injective and the analytic Fredholm theory \cite{colton2019inverse} then implies that the set of positive transmission eigenvalues, if non-empty, is discrete, with $+\infty$ as the only possible accumulation point.

The proof is complete.
\end{proof}

\begin{remark}
In Theorem \ref{thm_wp_1}, variational method and analytic Fredholm theory are employed to establish the discreteness of positive transmission eigenvalues for a general bounded Lipschitz domain $\Omega$ with a connected complement, with infinity as the only possible accumulation point. In next section, we specialize to the case of a radially symmetric domain, where the existence of a sequence of positive transmission eigenvalues  is proved and their asymptotic expressions  are derived in Theorem \ref{lem_bl_1}.
\end{remark}

\section{Boundary localization of transmission eigenfunctions} \label{section_3}

In this section, we study the EEITEP \eqref{eq_intro_EHuOmega} in a radially symmetric setting. By a scaling argument, we assume that $\Omega$ is the unit ball in $\mathbb{R}^3$, namely, $\Omega = \left\{\bm{x} \in \mathbb{R}^3; \left| \bm{x} \right| < 1\right\}$. Using a suitable ansatz for transmission eigenfunctions of \eqref{eq_intro_EHuOmega}, we first show the existence of a sequence of transmission eigenvalues of \eqref{eq_intro_EHuOmega} and derive their asymptotic behavior as they tend to infinity. We then investigate the boundary localization properties of the corresponding transmission eigenfunctions. More precisely, we prove that the electromagnetic components are boundary-localized in $\Omega$, whereas the elastic part does not exhibit such localization.

For $0<\tau<1$, let $\Omega_{\tau} = \left\{ \bm{x}\in\mathbb{R}^3; \left|\bm{x}\right| < \tau \right\}$. We now introduce the notion of boundary localization for a vector-valued function $\bm{\phi}$ in $\Omega$.

\begin{definition}\label{def:bd}
A vector-valued function $\bm{\varphi} \in L^2\left(\Omega\right)^3$ is said to be boundary-localized in $\Omega$, if it satisfies 
\begin{align} \label{ineq_bl_phi/phi<<1}
	\frac{\left\|\bm{\varphi}\right\|_{L^2\left(\Omega_{\tau}\right)^3}}{\left\|\bm{\varphi}\right\|_{L^2\left(\Omega\right)^3}} \ll 1,
\end{align}
for some $\tau\in\left(0,1\right)$ sufficiently close to $1$. 
\end{definition}

To analyze the existence and asymptotic behavior of transmission eigenvalues of the EEITEP \eqref{eq_intro_EHuOmega} in the radially symmetric setting, we make use of the distribution of the zeros of spherical Bessel functions and their derivatives. Let $J_n(x)$ denote the Bessel function of the first kind of order $n$. The spherical Bessel function is defined by 
\begin{align} \label{eq_bl_jn=Jn+12}
	j_n(x) := \sqrt{\frac{\pi}{2x}} J_{n+\frac{1}{2}}(x).
\end{align}
Let $ r_{n,s} $ and $ r^{\prime}_{n,s} $ denote the $s$-th positive zeros of $ j_n(x) $ and $ j^{\prime}_n(x) $ respectively. It follows from \cite{abramowitz1965handbook} that
\begin{align} \label{ineq_bl_n<rn}
	n+\frac{1}{2} \leq r^{\prime}_{n,1} < r_{n,1} < r^{\prime}_{n,2} < r_{n,2} < r^{\prime}_{n,3} < r_{n,3} < \cdots.
\end{align}
In the subsequent analysis, we focus on the scenario where the order $ n $ is sufficiently large. We define the sequences $\left\{s_1\left(n\right)\right\}_{n\in\mathbb{N}_+}$ and $\left\{s_2\left(n\right)\right\}_{n\in\mathbb{N}_+}$ as
\begin{align} \label{def_bl_s1ns2n}
	s_1\left(n\right) := \left\lfloor\left(n+\frac{1}{2}\right)^{\gamma_1}\right\rfloor, \quad s_2\left(n\right) := \left\lfloor\left(n+\frac{1}{2}\right)^{\gamma_2}\right\rfloor, 
\end{align}
where $ 0<\gamma_1<\gamma_2<1 $, and $ \lfloor h \rfloor $ stands for the integer part of $ h\in\mathbb{R} $.

 Let $(r, \theta, \phi)$ denote the spherical coordinates in $\mathbb{R}^3$. For $\bm{x} = (x_1, x_2, x_3)^\top \in \mathbb{R}^3$, the Cartesian components are given by $(r\sin\theta\cos\phi, r\sin\theta\sin\phi, r\cos\theta)^\top$, where $r > 0$, $\theta \in [0, \pi]$, and $\phi \in [0, 2\pi)$. We define the associated orthonormal basis vectors as follows:
\begin{align} \label{def_bl_rthetavarphi}
	\hat{\bm{r}}:=\begin{pmatrix}
		\sin\theta\cos\phi \\
		\sin\theta\sin\phi \\
		\cos\theta
	\end{pmatrix}, \quad  \hat{\bm{\theta}}:=\begin{pmatrix}
		\cos\theta\cos\phi \\
		\cos\theta\sin\phi \\
		-\sin\theta
	\end{pmatrix}, \quad \hat{\bm{\phi}}:=\begin{pmatrix}
		-\sin\phi \\
		\cos\phi \\
		0
	\end{pmatrix}.
\end{align}
It follows that $\hat{\bm{x}} = \hat{\bm{r}}$ and $\bm{x} = r\hat{\bm{r}}$. The spherical harmonics $Y^m_n(\theta, \phi)$ are defined by
\begin{align} \label{def_bl_Ynm}
	Y^m_n(\theta, \phi) := c_{n,m} P_n^{|m|}(\cos\theta) e^{\mathrm{i}m\phi}, 
\end{align}
where $P_n^{|m|}(x)$ denotes the associated Legendre functions of the first kind and 
\begin{align} \label{def_bl_cnm}
	c_{n,m} := \sqrt{\frac{2n+1}{4\pi} \frac{(n-|m|)!}{(n+|m|)!}}
\end{align}
is the normalization constant. Here, $n\in\mathbb{N}_0:=\{0\}\cup \mathbb N$ and $m \in \mathbb{Z}$ such that $-n \leq m \leq n$. Finally, the surface gradient operator $\nabla_{\mathrm{s}}$ on the unit sphere is defined as
\begin{align} \label{def_bl_surfacegrad}
	\nabla_{\mathrm{s}}f := \frac{\partial f}{\partial \theta} \hat{\bm{\theta}} + \frac{1}{\sin\theta} \frac{\partial f}{\partial \phi} \hat{\bm{\phi}}.
\end{align}

In this section, we consider the case where the domain $\Omega$ for the EEITEP \eqref{eq_intro_EHuOmega} is the unit ball. Our goal is to show the existence of a sequence of positive real transmission eigenvalues $\{k_n\}_{n \in \mathbb{N}}$ such that $\lim_{n \to \infty} k_n = \infty$. Let $\omega_n = k_n c_0$, where $c_0 > 0$ appears in \eqref{eq_intro_EHuOmega}. Following the definition of the shear wavenumber in \eqref{eq_intro_kpks}, we introduce
\begin{equation}\label{eq:ksn}
k_{\mathrm{s},n} = \frac{\omega_n}{c_{\mathrm{s}}},
\end{equation}
where $c_{\mathrm{s}}$ is given by \eqref{def_intro_cscp}. 
Following the approach in \cite{colton2019inverse, dassios1995elastic}, we employ the following ansatz for the remainder of this paper:
\begin{align} 
	\bm{E}_n^m\left(\bm{x}\right) & =  b_n^m j_n\left(k_nr\right) \nabla_{\mathrm{s}} Y^m_n\left(\theta,\phi\right) \times \hat{\bm{r}}, \label{def_bl_En} \\ 
	\bm{H}_n^m\left(\bm{x}\right) & = \frac{b_n^m}{\mathrm{i}k_n} \curl \left(j_n\left(k_n r\right) \nabla_{\mathrm{s}} Y^m_n\left(\theta,\phi\right) \times \hat{\bm{r}}\right), \label{def_bl_Hn_1} \\  
	\bm{u}_n^m\left(\bm{x}\right) & = \beta_n^m j_n\left(k_{\mathrm{s},n}r\right) \nabla_{\mathrm{s}} Y^m_n\left(\theta,\phi\right) \times \hat{\bm{r}}, \label{def_bl_un}
\end{align}
to represent the elastic-electromagnetic interior transmission eigenfunctions $(\bm{E}, \bm{H}, \bm{u}) \in H(\mathbf{curl}, \Omega) \times H(\mathbf{curl}, \Omega) \times H^1(\Omega)^3$ to  \eqref{eq_intro_EHuOmega}. Here, $(b_n^m, \beta_n^m)^\top \in \mathbb{C}^2 \setminus \{\bm{0}\}$ is a non-zero vector determined by the transmission conditions in \eqref{eq_intro_EHuOmega}.

In the following theorem,  we prove that for all sufficiently large $n$, there exists a transmission eigenvalue $k_n$ for the EEITEP \eqref{eq_intro_EHuOmega}, and we derive the asymptotic behavior of $k_n$ as $n\to\infty$.

\begin{theorem} \label{lem_bl_1}
Consider the EEITEP \eqref{eq_intro_EHuOmega} defined in the unit ball $\Omega \subset \mathbb{R}^3$. Suppose that $\rho > 0$ and $\lambda, \mu$ satisfy \eqref{eq_intro_strongconvexity}. Recall that $c_0 > 0$ appears in \eqref{eq_intro_EHuOmega} and $c_{\mathrm{s}}$ is  given by \eqref{def_intro_cscp}. Let $r_{n,s}$ denote the $s$-th positive zero of the spherical Bessel function $j_n(x)$, and let $s_i(n)$ for $i=1,2$ be defined as in \eqref{def_bl_s1ns2n}. Assume that
\begin{equation}\label{eq:delta cond}
    \delta := \frac{c_0}{c_{\mathrm{s}}} > 1. 
\end{equation}
Then, for all sufficiently large $n$, there exists a transmission eigenvalue $k_n$ such that
\begin{align} \label{eq_bl_kn_1}
	k_n \in \big(r_{n,s_1(n)}, r_{n,s_2(n)}\big), 
\end{align}
and for each $m=-n,\ldots,n$ the triplet $(\bm{E}_n^m, \bm{H}_n^m, \bm{u}_n^m)$ defined in \eqref{def_bl_En}-\eqref{def_bl_un} constitutes a transmission eigenfunction of \eqref{eq_intro_EHuOmega}. Furthermore, as $n \to \infty$, the following asymptotic expansion holds:
\begin{align} \label{eq_bl_kn_2}
	k_n = \left(n+\frac{1}{2}\right) \left[1 + c\left(n+\frac{1}{2}\right)^{-\xi} + O\left(n^{-\xi-1}\right)\right],
\end{align}
where $\xi \in [\xi_2, \xi_1]$ with $0 < \xi_i = \frac{2(1-\gamma_i)}{3} < \frac{2}{3}$ for $i=1,2$, and $c > 0$ is a constant independent of $n$.

\end{theorem}

\begin{proof}
To establish the existence of transmission eigenvalues, we assume the eigenfunctions take the form \eqref{def_bl_En}-\eqref{def_bl_un}. By \eqref{def_bl_En}, one can derive 
\begin{align} \label{eq_bl_nutimesEn}
	\bm{\nu}\times \bm{E}_n^m = b^m_n j_n\left(k_nr\right) \nabla_{\mathrm{s}} Y^m_n.
\end{align}
For $\bm{H}_n^m$ defined in \eqref{def_bl_Hn_1}, using
\begin{align} \label{eq_bl_sinYnmnn+1}
	\frac{1}{\sin\theta} \frac{\partial}{\partial\theta} \left(\sin\theta\frac{\partial Y^m_n}{\partial\theta}\right) + \frac{1}{\sin^2\theta} \frac{\partial^2Y^m_n}{\partial\phi^2} + n\left(n+1\right) Y^m_n = 0,
\end{align}
we obtain
\begin{align} \label{def_bl_Hn_2}
	\bm{H}_n^m 
	= & \frac{b_n^m}{\mathrm{i}k_nr} \left(n\left(n+1\right) j_n\left(k_n r\right) Y^m_n \hat{\bm{r}} + \frac{\partial\left(r j_n\left(k_n r\right)\right)}{\partial r}  \nabla_{\mathrm{s}} Y^m_n \right).
\end{align}
Consequently, it follows that
\begin{align} \label{eq_bl_nutimesHn}
	\frac{\mathrm{i}}{k_n} \bm{\nu}\times\bm{H}_n^m 
	= - \frac{b^m_n}{k_n^2r}  \frac{\partial\left(rj_n\left(k_nr\right)\right)}{\partial r} \nabla_{\mathrm{s}} Y^m_n \times \hat{\bm{r}}.
\end{align}
Regarding $\bm{u}_n^m$ formulated as \eqref{def_bl_un}, it is straightforward to verify that $\nabla\cdot \bm{u}_n^m\left(\bm{x}\right) = 0$. Calculating the radial components of the symmetric gradient yields
\begin{align*} 
	\left(\nabla^{\mathrm{s}} \bm{u}_n^m\right) \bm{\nu} = \frac{1}{2}\left(\frac{\partial \bm{u}_n^m}{\partial r} - \frac{\bm{u}_n^m}{r}\right) = \frac{r\beta^m_n}{2} \frac{\partial}{\partial r} \left(\frac{j_n\left(k_{\mathrm{s},n}r\right)}{r}\right) \nabla_{\mathrm{s}} Y^m_n \times \hat{\bm{r}}.
\end{align*}
Then we have
\begin{align} \label{eq_bl_Tnuun}
	T_{\bm{\nu}} \bm{u}_n^m = \mu r \beta^m_n \frac{\partial}{\partial r} \left(\frac{j_n\left(k_{\mathrm{s},n}r\right)}{r}\right) \nabla_{\mathrm{s}} Y^m_n \times \hat{\bm{r}}.
\end{align}
Similarly, we deduce that
\begin{align} \label{eq_bl_nutimesun}
	\bm{\nu}\times \bm{u}_n^m = \beta_n^m j_n\left(k_{\mathrm{s},n}r\right) \nabla_{\mathrm{s}} Y^m_n.
\end{align}

Substituting \eqref{eq_bl_nutimesEn}-\eqref{eq_bl_nutimesun} into the transmssion  conditions of \eqref{eq_intro_EHuOmega}, it follows that
\begin{align*}
\begin{cases}
	-\frac{1}{k_n^2r} b^m_n \frac{\partial}{\partial r} \left(rj_n\left(k_nr\right)\right)\nabla_{\mathrm{s}} Y^m_n \times \hat{\bm{r}} - \mu r \beta^m_n \frac{\partial}{\partial r} \left(\frac{j_n\left(k_{\mathrm{s},n}r\right)}{r}\right) \nabla_{\mathrm{s}} Y^m_n \times \hat{\bm{r}} = 0, \\
	b^m_n j_n\left(k_nr\right)\nabla_{\mathrm{s}} Y^m_n - \beta_n^m j_n\left(k_{\mathrm{s},n}r\right) \nabla_{\mathrm{s}} Y^m_n = 0.
\end{cases}
\end{align*}
Therefore it yields 
$$
\begin{cases}
	-\frac{1}{k_n^2r} b^m_n \frac{\partial}{\partial r} \left(rj_n\left(k_nr\right)\right) - \mu r \beta^m_n \frac{\partial}{\partial r} \left(\frac{j_n\left(k_{\mathrm{s},n}r\right)}{r}\right)  = 0, \\
	b^m_n j_n\left(k_nr\right) - \beta_n^m j_n\left(k_{\mathrm{s},n}r\right) = 0.
\end{cases}
$$
Note that $\left(b^m_n, \beta^m_n\right)^\top \in \mathbb C^2 \setminus \{\bf 0\}$. The determinant of the coefficient matrix must vanish. This leads to 
\begin{align*}
	\frac{1}{k_n^2}\frac{\partial}{\partial r} \left(r j_n\left(k_nr\right)\right) j_n\left(k_{\mathrm{s},n} r\right) + \mu r^2 \frac{\partial}{\partial r}\left(\frac{j_n\left(k_{\mathrm{s},n} r\right)}{r}\right)j_n\left(k_n r\right) = 0.
\end{align*}
Applying the recurrence relation 
\begin{align*} 
	\frac{\d j_n\left(kr\right)}{\d r} = \frac{n}{r} j_n\left(kr\right) - kj_{n+1}\left(kr\right),
\end{align*}
it yields that 
\begin{align*}
	\left(\frac{n+1}{k_n^2} + \mu\left(n-1\right)\right) j_n\left(k_nr\right) j_n\left(k_{\mathrm{s},n}r\right) & - \frac{r}{k_n} j_{n+1}\left(k_nr\right) j_n\left(k_{\mathrm{s},n}r\right) \\
	&  - \mu k_{\mathrm{s},n} r j_{n+1} \left(k_{\mathrm{s},n} r\right) j_n\left(k_nr\right) = 0.
\end{align*}
Let $r = 1$. Then $k_n $ is a positive zero of the function
\begin{align*}
	f_n\left(k\right) = & \left(n+1+\mu k^2\left(n-1\right)\right) j_n\left(k\right) j_n\left(\delta k\right) - k j_{n+1}\left(k\right)j_n\left(\delta k\right) \\
	& - \mu\delta k^3 j_{n+1}\left(\delta k\right)j_n\left(k\right),
\end{align*}
where $\delta$ is given by \eqref{eq:delta cond} with  $\delta > 1$. Using \eqref{eq_bl_jn=Jn+12}, we have
\begin{align*}
	g_n\left(k\right) = & \frac{2k\sqrt{\delta}}{\pi} f_n\left(k\right) \\
	= & \left(n+1+\mu k^2\left(n-1\right)\right) J_{n+\frac{1}{2}}\left(k\right) J_{n+\frac{1}{2}}\left(\delta k\right) - k J_{n+\frac{3}{2}}\left(k\right) J_{n+\frac{1}{2}}\left(\delta k\right) \\
	& - \mu\delta k^3 J_{n+\frac{3}{2}}\left(\delta k\right) J_{n+\frac{1}{2}}\left(k\right).
\end{align*}
Consequently, it follows that 
\begin{align*}
	& \frac{1}{r_{n,s_1\left(n\right)} r_{n,s_2\left(n\right)}}g_n\left(r_{n,s_1\left(n\right)}\right) g_n\left(r_{n,s_2\left(n\right)}\right) \\
	= &  J_{n+\frac{3}{2}}\left(r_{n,s_1\left(n\right)}\right) J_{n+\frac{1}{2}}\left(\delta r_{n,s_1\left(n\right)}\right) J_{n+\frac{3}{2}}\left(r_{n,s_2\left(n\right)}\right) J_{n+\frac{1}{2}}\left(\delta r_{n,s_2\left(n\right)}\right).
\end{align*}

To establish the existence of eigenvalues satisfying \eqref{eq_bl_kn_1}, it suffices to verify $g_n\left(r_{n,s_1\left(n\right)}\right) g_n\left(r_{n,s_2\left(n\right)}\right)<0$. According to \cite{korenev2002bessel}, for $x>n$ with $n$ sufficiently large, the Bessel function admits the asymptotic expansion
\begin{align} \label{eq_bl_Jn(x)_1}
	J_n(x) =  \left(\frac{2}{\pi \sqrt{x^2-n^2}}\right)^{1/2} \cos \left(\sqrt{x^2-n^2} - \frac{n \pi}{2} + n\arcsin\left(\frac{n}{x}\right) - \frac{\pi}{4}\right)\left(1+o\left(1\right)\right).
\end{align}
Note that
\begin{align*}
	\arcsin x = \frac{\pi}{2} - \sqrt{2\left(1-x\right)} - O\left(\left(1-x\right)^{\frac{3}{2}}\right),
\end{align*}
as $x\to1^-$. For $\frac{n}{x_n}\to1^-$, \eqref{eq_bl_Jn(x)_1} implies 
\begin{align} \label{eq_bl_Jn(x)_2}
	J_n\left(x_n\right) & = \left(\frac{2}{\pi \sqrt{x_n^2-n^2}}\right)^{1/2} \cos\left(\sqrt{x_n^2-n^2} - n\sqrt{2\left(1-\frac{n}{x_n}\right)} - \frac{\pi}{4}\right) \left(1+o\left(1\right)\right) \notag \\
	& = \left(\frac{2}{\pi \sqrt{x_n^2-n^2}}\right)^{1/2} \cos\left(\frac{3}{2\sqrt{2}} \frac{\left(x_n-n\right)^{\frac{3}{2}}}{\sqrt{n}} + O\left(\frac{\left(x_n-n\right)^{\frac{5}{2}}}{n^{\frac{3}{2}}}\right)- \frac{\pi}{4}\right) \left(1+o\left(1\right)\right),
\end{align}
where $n$ is large. According to \cite{qu1999best}, the $s$-th positive root $r_{n,s}$ of $j_n(x)$ has its upper and lower bounds
\begin{align*}
	n + \frac{1}{2} - \frac{a_s}{2^{1/3}} \left(n+\frac{1}{2}\right)^{1/3} < r_{n, s} < n + \frac{1}{2} - \frac{a_s}{2^{1/3}} \left(n+\frac{1}{2}\right)^{1/3} + \frac{3}{20} a_s^2\left(\frac{2}{n+\frac{1}{2}}\right)^{1/3},
\end{align*}
where $a_s$ is the $s$-th negative zero of the Airy function $Ai(x)$ satisfying
\begin{align*}
	a_s = - \left(\frac{3\pi}{8}\left(4s-1\right)\right)^{\frac{2}{3}} \left(1+\sigma_s\right), \quad \left|\sigma_s\right|\leq 0.130\left(\frac{3\pi}{8}\left(4s-1.051\right)\right)^{-2}.
\end{align*}
For $i=1,2$, this implies
\begin{align} \label{eq_bl_r{n,si(n)}}
	& \left(n+\frac{1}{2}\right) \left[1+\frac{\left(3\pi\right)^{\frac{2}{3}}}{2} \left(n+\frac{1}{2}\right)^{-\xi_i} + O\left(\left(n+\frac{1}{2}\right)^{\frac{\xi_i}{2}-1}\right)\right] < r_{n,s_i\left(n\right)} \notag \\
	< & \left(n+\frac{1}{2}\right) \left[1+\frac{\left(3\pi\right)^{\frac{2}{3}}}{2} \left(n+\frac{1}{2}\right)^{-\xi_i} + O\left(\left(n+\frac{1}{2}\right)^{\frac{\xi_i}{2}-1}\right)  + O\left(\left(n+\frac{1}{2}\right)^{-2\xi_i}\right)\right].
\end{align}
Consequently, we have $\frac{n+\frac{3}{2}}{r_{n,s_i\left(n\right)}}\to1^-$ for $i=1,2$. Combining \eqref{eq_bl_Jn(x)_2} and \eqref{eq_bl_r{n,si(n)}}, it follows that
\begin{align*}
	J_{n+\frac{3}{2}}\left(r_{n,s_i\left(n\right)}\right) = \left(\frac{2}{\pi} N_{n,1,i}\right)^{\frac{1}{2}} \cos\left(N_{n,2,i} - \frac{\pi}{4}\right) \left(1+o\left(1\right)\right), \quad i=1,2,
\end{align*}
where $N_{n,1,i} = \left(\left(3\pi\right)^{\frac{1}{3}}\left(n+\frac{1}{2}\right)^{1-\frac{\xi_i}{2}} + o\left(n^{1-\frac{\xi_i}{2}}\right)\right)^{-1}$ and $N_{n,2,i} = \frac{9\pi}{8}\left(n+\frac{1}{2}\right)^{1-\frac{3}{2}\xi_i} + o\left(n^{1-\frac{3}{2}\xi_i}\right)$. Similarly, we have
\begin{align*}
	J_{n+\frac{1}{2}}\left(\delta  r_{n,s_i\left(n\right)}\right) = \left(\frac{2}{\pi} M_{n,1,i}\right)^{\frac{1}{2}} \cos\left(M_{n,2,i} - \frac{\pi}{4} \right) \left(1+o\left(1\right)\right), \quad i=1,2,
\end{align*}
where $M_{n,1,i} = \left(\left(n+\frac{1}{2}\right)\sqrt{\delta^2-1}\right)^{-1}$ and $M_{n,2,i} = \left(n+\frac{1}{2}\right)\left(\sqrt{\delta^2-1} - \frac{\pi}{2} +\arcsin\frac{1}{\delta}\right) + o\left(n\right)$. Suppose that
\begin{align*}
	J_{n+\frac{3}{2}}\left(r_{n,s_1\left(n\right)}\right) J_{n+\frac{1}{2}}\left(\delta r_{n,s_1\left(n\right)}\right)>0.
\end{align*} 
For all $\delta>1$, it can be proved that $\sqrt{\delta^2-1} - \frac{\pi}{2} +\arcsin\frac{1}{\delta}$ never vanishes. Consequently, for the product $\cos\left(N_{n,2,2} - \frac{\pi}{4}\right) \cos\left( M_{n,2,2}-\frac{\pi}{4} \right)$, since the cosine functions do not have the same frequency, it can be obtained that there exists $\xi_2$ such that the product $\cos\left(N_{n,2,2} - \frac{\pi}{4}\right) \cos\left( M_{n,2,2}-\frac{\pi}{4} \right)$ is negative. Hence we have proved the existence of $k_n$ satisfying \eqref{eq_bl_kn_1}. The asymptotic expression \eqref{eq_bl_kn_2} follows directly from \eqref{eq_bl_r{n,si(n)}}.
	
The proof is complete.
\end{proof}

\begin{remark}
 The assumption \eqref{eq:delta cond} is a technical condition that is readily fulfilled in most physical scenarios. In fact, the EEITEP \eqref{eq_intro_EHuOmega} arises from the study of invisibility phenomena in the direct scattering problem \eqref{eq_intro_EE_R3}. Consequently, the parameters $c_0$, $\lambda$, $\mu$, and $\rho$ possess specific physical meanings derived from the scattering model \eqref{eq_intro_EE_R3}. When the elastic body $\Omega$ becomes invisible, the corresponding transmission eigenvalue $k$ is given by \eqref{eq:dfn k}. In this context, $c_0$ in \eqref{eq_intro_EHuOmega} refers to the wave speed in the homogeneous electromagnetic background, while $c_{\mathrm{s}}$, defined in \eqref{def_intro_cscp}, represents the velocity of the shear wave propagating within the elastic body $\Omega$. In a typical physical setup, $c_{\mathrm{s}}$ is significantly smaller than $c_0$. Therefore, the assumption \eqref{eq:delta cond} is generally satisfied from a physical perspective.

\end{remark}

Since the existence and asymptotic behavior of the transmission eigenvalues in Theorem \ref{lem_bl_1} have been established, we now turn to investigating the localization and non-localization properties of the corresponding transmission eigenfunctions. Specifically, in the following theorem, we show that for sufficiently large $n$, the electromagnetic components exhibit boundary localization, whereas the elastic component does not display such concentration.

\begin{theorem} \label{thm_bl}
Under the same setup as in Theorem \ref{lem_bl_1}, if the assumption \eqref{eq:delta cond} is satisfied, let $k_n$ be the transmission eigenvalue of the EEITEP \eqref{eq_intro_EHuOmega} satisfying \eqref{eq_bl_kn_2}, and let $(\bm{E}_n^m, \bm{H}_n^m, \bm{u}_n^m)$ be the corresponding transmission eigenfunction triplet defined in \eqref{def_bl_En}-\eqref{def_bl_un}.   For all sufficiently large $n$, in the case where $b_n^m \neq 0$ and $\beta_n^m \neq 0$, the following properties hold in the sense of Definition \ref{def:bd}:
\begin{enumerate}
	\item the transmission eigfunctions $\bm{E}_n^m$ and $\bm{H}_n^m$ are boundary-localized in $\Omega$;
	\item the transmission eigfunction $\bm{u}_n^m$ is not boundary-localized in $\Omega$,
\end{enumerate}
where $m\in \{0,\pm 1, \ldots, \pm n\}$. 
\end{theorem}

\begin{proof}
To prove the above assertions, using Definition \ref{def:bd},  we only need to establish the following asymptotic properties: for any fixed $\tau\in\left(0,1\right)$ sufficiently close to $1$, it follows that
\begin{align*}
	\frac{\|\bm{E}_n^m\|_{L^2(\Omega_{\tau})^3}}{\|\bm{E}_n^m\|_{L^2(\Omega)^3}}\to 0, \quad \frac{\|\bm{H}_n^m\|_{L^2(\Omega_{\tau})^3}}{\|\bm{H}_n^m\|_{L^2(\Omega)^3}}\to 0, \quad \text{as } n\to\infty,
\end{align*}
whereas there exists a $\tau\in (0, 1)$ sufficiently close to $1$ such that
\begin{align*}
	\frac{\|\bm{u}_n^m\|_{L^2(\Omega_{\tau})^3}}{\|\bm{u}_n^m\|_{L^2(\Omega)^3}} \nrightarrow 0, \quad \text{as } n\to\infty.
\end{align*}
Since the analysis of $\bm{E}_n^m$ is analogous to that of $\bm{H}_n^m$, we omit the details for $\bm{E}_n^m$ and focus only on the proofs for $\bm{H}_n^m$ and $\bm{u}_n^m$.

For any $\tau<1$ sufficiently close to $1$, let $\Omega_{\tau} = \left\{\bm{x}; |\bm{x}| < \tau\right\}$. Observe that $H_n^m\left(\bm{x}\right)$ can be rewritten as \eqref{def_bl_Hn_2}. Using the recurrence relation
\begin{align*} 
	j_n^{\prime}\left(k_n r\right) = \frac{n j_{n-1}\left(k_n r\right) - \left(n+1\right) j_{n+1}\left(k_n r\right)}{2n+1},
\end{align*}
and the asymptotic expression \eqref{eq_bl_kn_2} of $k_n$ for $n\geq N$, we have
\begin{align*}
	\bm{H}_n^m\left(\bm{x}\right) = \frac{b_n^m}{\mathrm{i}} & \left(\frac{n\left(n+1\right)}{k_nr} j_n\left(k_n r\right) Y^m_n\left(\theta,\phi\right) \hat{\bm{r}} \right.\\
	& + \left.\frac{n j_{n-1}\left(k_n r\right) - \left(n+1\right) j_{n+1}\left(k_n r\right)}{2n+1} \nabla_{\mathrm{s}} Y^m_n\left(\theta,\phi\right) \left(1+o\left(\frac{1}{n}\right)\right)\right),
\end{align*}
for sufficiently large $n$. Orthogonality properties
\begin{align} \label{eq_bl_YYgradYgradY}
	\int_{\mathbb{S}^2} Y^m_n \overline{Y^{m^{\prime}}_{n^{\prime}}} \d s= \delta_{mm^{\prime}}\delta_{nn^{\prime}}, \quad \int_{\mathbb{S}^2} \nabla_{\mathrm{s}} Y^m_n \cdot \nabla_{\mathrm{s}} \overline{Y^{m^{\prime}}_{n^{\prime}}} \d s= n\left(n+1\right)\delta_{mm^{\prime}}\delta_{nn^{\prime}},
\end{align}
yield
\begin{align*} 
	& \left\|\bm{H}_n^m\left(\bm{x}\right)\right\|^2_{L^2\left(\Omega_{\tau}\right)^3} = \lvert b_n^m\rvert^2 \int_{0}^{\tau} \left(\frac{n^2\left(n+1\right)^2}{{k_n}^2} j_n^2\left(k_n r\right) \right.\notag\\
	& \quad \left.+ n\left(n+1\right) \left(\frac{n j_{n-1}\left(k_n r\right) - \left(n+1\right) j_{n+1}\left(k_n r\right)}{2n+1}\right)^2 \left(1+o\left(\frac{1}{n}\right)\right) r^2\right) \d r \notag\\
	& \leq  \lvert b_n^m\rvert^2 \int_{0}^{\tau} \left(\frac{n^2\left(n+1\right)^2}{{k_n}^2} j_n^2\left(k_n r\right) + n\left(n+1\right) \left(j_{n-1}^2\left(k_n r\right) + j_{n+1}^2\left(k_n r\right)\right) r^2\right) \d r.
\end{align*}
For large enough $n$, the inequality $k_n \tau < n-\frac{1}{2}$ holds for any fixed $\tau\in\left(0,1\right)$. From \eqref{ineq_bl_n<rn} we know that $j_{n-1}\left(r\right)$ and $j_{n+1}\left(r\right)$ increase monotonically on $\left(0,k_n\tau\right)$. Consequently, $0<j_{n-1}\left(s\right)<j_{n-1}\left(k_n\tau\right)$ and $0<j_{n+1}\left(s\right)<j_{n+1}\left(k_n\tau\right)$ for all $0<s<k_n\tau<n-\frac{1}{2}$. According to \cite{korenev2002bessel,watson1922treatise}, for sufficiently large real $\nu$ and fixed $0<r<1$, it follows that
\begin{align} \label{def_bl_Jnnr}
	J_{\nu}\left(\nu r\right)=\frac{r^\nu \mathrm{e}^{\nu\sqrt{1-r^2}}}{\left(2 \pi \nu\right)^{\frac{1}{2}}\left(1-r^2\right)^{\frac{1}{4}}\left(1+\sqrt{1-r^2}\right)^\nu}\left(1+o\left(1\right)\right).
\end{align}
By direct calculation, we have
\begin{align} \label{ineq_bl_re/1+1-r2}
	\frac{r\mathrm{e}^{\sqrt{1-r^2}}}{1+\sqrt{1-r^2}} < 1, \quad \text{for all } 0<r<1.
\end{align}
It follows that
\begin{align*}
	j_n\left(k_n \tau\right) = \sqrt{\frac{\pi}{2k_n \tau}} J_{n+\frac{1}{2}}\left(k_n \tau\right) < \frac{1}{2\sqrt{\left(n+\frac{1}{2}\right)k_n \tau}} \frac{1}{\left(1-\left(\frac{k_n \tau}{n+\frac{1}{2}}\right)^2\right)^{\frac{1}{4}}} < \frac{c}{n+\frac{1}{2}},
\end{align*}
where $c$ is a constant independent of $n$. Similar estimates hold for $j_{n-1}\left(k_n \tau\right)$ and $j_{n+1}\left(k_n \tau\right)$. Then we deduce
\begin{align} \label{est_bl_Hn_Omegatau_2}
	& \left\|\bm{H}_n^m\left(\bm{x}\right)\right\|^2_{L^2\left(\Omega_{\tau}\right)^3} \notag\\
	\leq & \lvert b_n^m\rvert^2 \left(\frac{n^2\left(n+1\right)^2}{{k_n}^2} j_n^2\left(k_n \tau\right)\tau + n\left(n+1\right) \left(j_{n-1}^2\left(k_n \tau\right) + j_{n+1}^2\left(k_n \tau\right)\right) \tau^3\right) < \lvert b_n^m\rvert^2 c,
\end{align}
where $c$ is a constant independent of $n$. Let $\tau_1 = \frac{r_{n,1}^{\prime}}{k_n}$. From \eqref{eq_bl_kn_1}, we have $\tau_1<1$ and $n+\frac{1}{2} < r_{n,1}^{\prime} = k_n\tau_1$. Utilizing \eqref{eq_bl_kn_2} and \eqref{eq_bl_Jn(x)_1}, one can derive 
\begin{align*}
	& \int_{\tau_1}^{1} j_n^2\left(k_n r\right) r \d r \\
	= & \int_{\tau_1}^{1} \frac{1}{k_n k_{n,r}} \cos^2 \left(k_{n,r} - \frac{\left(n+\frac{1}{2}\right) \pi}{2} + \left(n+\frac{1}{2}\right)\arcsin\left(\frac{n+\frac{1}{2}}{k_n r}\right) - \frac{\pi}{4}\right) \left(1+o\left(1\right)\right) \d r \\
	\geq & \frac{1}{4k_n k_{n,1}} \int_{\tau_1}^{1}  \left(1+\sin 2\left(k_{n,r} - \frac{\left(n+\frac{1}{2}\right) \pi}{2} + \left(n+\frac{1}{2}\right)\arcsin\left(\frac{n+\frac{1}{2}}{k_n r}\right)\right)\right) \d r \\
	= & \frac{1-\tau_1}{4k_n k_{n,1}} \left(1+o\left(1\right)\right) \geq c \left(n+\frac{1}{2}\right)^{-2+\frac{\xi}{2}},
\end{align*}
where $k_{n,r} = \sqrt{k_n^2r^2 - \left(n+\frac{1}{2}\right)^2}$, $n$ is sufficiently large and $c$ is a constant independent of $n$. Thus we have
\begin{align} \label{est_bl_Hn_Omega}
	\left\|\bm{H}_n^m\left(\bm{x}\right)\right\|^2_{L^2\left(\Omega\right)^3} \geq \lvert b_n^m\rvert^2 \int_{\tau_1}^{1} \frac{n^2\left(n+1\right)^2}{{k_n}^2 } j_n^2\left(k_n r\right) r \d r \geq \lvert b_n^m\rvert^2 c \left(n+\frac{1}{2}\right)^{\frac{\xi}{2}}.
\end{align}
Combining \eqref{est_bl_Hn_Omegatau_2} and \eqref{est_bl_Hn_Omega} yields
\begin{align*}
	\frac{\left\|\bm{H}_n^m\left(\bm{x}\right)\right\|^2_{L^2\left(\Omega_{\tau}\right)^3}}{\left\|\bm{H}_n^m\left(\bm{x}\right)\right\|^2_{L^2\left(\Omega\right)^3}} \leq c n^{-\frac{\xi}{2}}.
\end{align*}
As $n \to \infty$, this ratio vanishes, demonstrating that 
\begin{align} \label{eq_bl_def2}
\lim_{n\rightarrow\infty}\frac{\|\bm{H}_n^m\|_{L^2(\Omega_{\tau})^3}}{\|\bm{H}_n^m\|_{L^2(\Omega)^3}}=0,
\end{align}
and therefore $\bm{H}_n^m$ is boundary-localized in $\Omega$ for sufficient large $n$.

Next, we show that the sequence $\left\{\bm{u}_n^m\left(\bm{x}\right)\right\}_{n=1}^\infty$ defined in \eqref{def_bl_un} is not boundary-localized in $\Omega$ as $n\rightarrow \infty$. Note that $j_n\left(\delta k_n r\right)$ reaches its maximum value at $r = \frac{r_{n,1}^{\prime}}{\delta k_n} < \frac{1}{\delta} < 1$. Since $\tau\in(0,1)$ is sufficiently close to $1$, it follows that $\tau$ satisfies $\frac{1}{\delta} < \tau < 1$. Therefore 
$$
\delta \tau >1. 
$$
For sufficiently large $n$, we deduce that
\begin{align} \label{eq_bl_jn2r2_1}
	\int_0^1 j_n^2 \left(\delta k_nr\right) r^2 \d r & = \frac{1}{\delta^3} \int_0^{\frac{r_{n,1}^{\prime}}{ k_n}} j_n^2 \left(k_nr\right) r^2 \d r + \frac{1}{\delta^3} \int_{\frac{r_{n,1}^{\prime}}{ k_n}}^{\delta} j_n^2 \left(k_nr\right) r^2 \d r \notag\\
	& \leq \frac{2}{\delta^3} \int_{\frac{r_{n,1}^{\prime}}{ k_n}}^{\delta} j_n^2 \left(k_nr\right) r^2 \d r = 2 \int_{\frac{r_{n,1}^{\prime}}{\delta k_n}}^{1} j_n^2 \left(\delta k_nr\right) r^2 \d r.
\end{align}
For $\tau\leq t\leq1$, utilizing \eqref{eq_bl_Jn(x)_1}, we obtain
\begin{align} \label{eq_bl_jn2r2_2}
	& \int_{\frac{r_{n,1}^{\prime}}{\delta k_n}}^{t} j_n^2\left(\delta k_n r\right) r^2 \d r \notag\\
	\sim & \frac{1}{\delta k_n} \int_{\frac{r_{n,1}^{\prime}}{\delta k_n}}^{t} \frac{r}{k_{\delta,n,r}} \cos^2 \left(k_{\delta,n,r} - \frac{\left(n+\frac{1}{2}\right)\pi}{2} + \left(n+\frac{1}{2}\right)\arcsin\theta_{\delta,n,r} - \frac{\pi}{4}\right) \d r \notag \\
	\sim & \frac{1}{2\delta k_n} \int_{\frac{r_{n,1}^{\prime}}{\delta k_n}}^{t} \frac{r}{k_{\delta,n,r}}\d r = \frac{1}{2\delta^3k_n^3} \left(k_{\delta,n,t}-\sqrt{{r_{n,1}^{\prime}}^2 - \left(n+\frac{1}{2}\right)^2}\right),
\end{align}	
where $k_{\delta,n,r} = \sqrt{\left(\delta k_n r\right)^2 - \left(n+\frac{1}{2}\right)^2}$ and $\theta_{\delta,n,r} = \frac{n+1/2}{\delta k_n r}$. Combining \eqref{eq_bl_jn2r2_1} and \eqref{eq_bl_jn2r2_2}, we arrive at 
\begin{align*}
	\frac{\left\|\bm{u}_n^m\left(\bm{x}\right)\right\|^2_{L^2\left(\Omega_{\tau}\right)^3}}{\left\|\bm{u}_n^m\left(\bm{x}\right)\right\|^2_{L^2\left(\Omega\right)^3}} & = \frac{{\beta_n^m}^2 \int_0^{\tau} j_n^2\left(\delta k_n r\right) r^2 \d r \int_{\mathbb{S}} |\nabla_{\mathrm{s}} Y^m_n\left(\theta,\phi\right) \times \hat{\bm{r}}|^2 \d s}{{\beta_n^m}^2 \int_0^1 j_n^2\left(\delta k_n r\right) r^2 \d r \int_{\mathbb{S}} |\nabla_{\mathrm{s}} Y^m_n\left(\theta,\phi\right) \times \hat{\bm{r}}|^2 \d s} \\
	& \geq \frac{\int_{\frac{r_{n,1}^{\prime}}{\delta k_n}}^{\tau} j_n^2\left(\delta k_n r\right) r^2 \d r}{2\int_{\frac{r_{n,1}^{\prime}}{\delta k_n}}^{1} j_n^2\left(\delta k_n r\right) r^2 \d r} \\
	& \sim \frac{k_{\delta,n,\tau}-\sqrt{{r_{n,1}^{\prime}}^2 - \left(n+\frac{1}{2}\right)^2}}{2 \left(k_{\delta,n,1}-\sqrt{{r_{n,1}^{\prime}}^2 - \left(n+\frac{1}{2}\right)^2}\right)} \\
	& \to \frac{\sqrt{\delta^2 \tau^2 - 1}}{2 \sqrt{\delta^2 - 1}} > 0, \quad n\to \infty.
\end{align*}
This implies that the sequence $\left\{\bm{u}_n^m\left(\bm{x}\right)\right\}_{n=1}^\infty$ is not boundary-localized in $\Omega$ as $n\rightarrow \infty$.

The proof is complete.
\end{proof}

\begin{remark}
	Theorem \ref{thm_bl} shows that the boundary localization property of $\bm{E}_n^m$ and $\bm{H}_n^m$ holds uniformly with respect to the order $m$. In the next section, we shall establish asymptotic lower bounds for $\|\nabla \bm{E}_n^m\|_{L^\infty}$ and $\|\nabla \bm{H}_n^m\|_{L^\infty}$, revealing that these estimates explicitly depend on $m$.
\end{remark}

\section{Gradient blow-up of the electromagnetic component} \label{section_4}

In this section, we investigate the gradient behavior of the electromagnetic components of the transmission eigenfunctions for the EEITEP \eqref{eq_intro_EHuOmega}. In the previous section, we proved that the electromagnetic components are boundary-localized, whereas the elastic component does not exhibit such concentration. Motivated by this distinction, we focus here on the electromagnetic fields and analyze the blow-up of their gradients near the boundary $\partial\Omega$. In particular, for each fixed $m$, we show that the blow-up rate depends quantitatively on both $m$ and $n$ as $n\to\infty$. 

\begin{theorem} \label{thm4.1}
Consider the same setup as in Theorem \ref{lem_bl_1}. Assume that the condition \eqref{eq:delta cond} holds. Let $k_n$ be the transmission eigenvalue of the EEITEP \eqref{eq_intro_EHuOmega} satisfying \eqref{eq_bl_kn_2}, and let $\left(\bm{E}_n^m, \bm{H}_n^m\right)$ be the corresponding electromagnetic components defined in \eqref{def_bl_En} and \eqref{def_bl_Hn_1}. For all sufficiently large $n$ and in the case where $b_n^m \neq 0$, for any fixed $\tau \in (0,1)$ sufficiently close to $1$, the estimates \eqref{est_grad_gradEn0}--\eqref{est_grad_gradHnm} hold for each fixed $m \in \{0,\pm 1, \ldots, \pm n\}$.
 If $m=0$, then
\begin{align} 
	\frac{\left\|\nabla\bm{E}_n^0\right\|_{L^{\infty}\left(\Omega\setminus\overline{\Omega_{\tau}}\right)^{3\times3}}}{\left\|\bm{E}_n^0\right\|_{L^2\left(\Omega\right)^3}} & \geq \frac{\Gamma\left(\frac{1}{3}\right)}{2^{7/6} 3^{1/2} \pi^{11/6}} n^{\frac{7}{6}+\frac{\xi_2}{2}} \left(1+o\left(1\right)\right), \label{est_grad_gradEn0} \\
	\frac{\left\|\nabla\bm{H}_n^0\right\|_{L^{\infty}\left(\Omega\setminus\overline{\Omega_{\tau}}\right)^{3\times3}}}{\left\|\bm{H}_n^0\right\|_{L^2\left(\Omega\right)^3}} & \geq \frac{\Gamma\left(\frac{1}{3}\right) \left(\sqrt{3} - 1\right) }{2^{\frac{13}{6}} 3^{\frac{1}{4}} \pi^{\frac{7}{3}} 5^{\frac{1}{2}}}  n^{\frac{2}{3}+\frac{\xi_2}{2}} \left(1+o\left(1\right)\right). \label{est_grad_gradHn0} 
\end{align}
If $m\neq0$, then 
\begin{align} 
	\frac{\left\|\nabla\bm{E}_n^m\right\|_{L^{\infty}\left(\Omega\setminus\overline{\Omega_{\tau}}\right)^{3\times3}}}{\left\|\bm{E}_n^m\right\|_{L^2\left(\Omega\right)^3}} & \geq \frac{\Gamma\left(\frac{1}{3}\right)}{(1.11)^{\frac{1}{2}} 2^{\frac{13}{6}} 3^{\frac{1}{2}} \pi^{\frac{11}{6}}} \frac{n^{\frac{7}{6}+\frac{\xi_2}{2}}}{\sqrt{\left|m\right|}} \left(1+o\left(1\right)\right), \label{est_grad_gradEnm} \\
	\frac{\left\|\nabla\bm{H}_n^m\right\|_{L^{\infty}\left(\Omega\setminus\overline{\Omega_{\tau}}\right)^{3\times3}}}{\left\|\bm{H}_n^m\right\|_{L^2\left(\Omega\right)^3}} & \geq \frac{\Gamma\left(\frac{1}{3}\right)}{(1.11)^{\frac{3}{2}} 2^{\frac{8}{3}} 3^{\frac{1}{2}} \pi^{\frac{11}{6}} 5^{\frac{1}{2}}} \frac{n^{\frac{7}{6}+\frac{\xi_2}{2}}}{\sqrt{\left|m\right|}} \left(1+o\left(1\right)\right), \label{est_grad_gradHnm}
\end{align}
where $0<\xi_2=\frac{2\left(1-\gamma_2\right)}{3}<\frac{2}{3}$.
\end{theorem}

\begin{proof}
The proof is divided into five steps. In Steps 1--4, we derive the lower bounds \eqref{est_grad_gradHn0} and \eqref{est_grad_gradHnm} for $\bm{H}_n^m$, which include an upper bound for its $L^2$-norm, a decomposition of its gradient, and lower bounds for the $L^{\infty}$-norm of its gradient in the cases $m = 0$ and $m \neq 0$, respectively. The corresponding estimates \eqref{est_grad_gradEn0} and \eqref{est_grad_gradEnm} for $\bm{E}_n^m$ are established in Step 5 by an analogous argument, where we provide only the key ingredients and a sketch of the proof.

\medskip

\textit{Step 1: Upper bounds for the $L^2$-norm of $\bm{H}_n^m$.}
Recall that $\bm{H}_n^m$ is given by \eqref{def_bl_Hn_2} with $b_n^m \neq 0$. Using \eqref{eq_bl_jn=Jn+12} and \eqref{eq_bl_kn_2}, we have 
\begin{align*} 
	& \int_0^1 j_{n+1}^2\left(k_n r\right) r^2 \d r < 
	\frac{\pi}{2k_n} \int_0^1 \frac{1}{r} J_{n+\frac{3}{2}}^2\left(k_n r\right) \d r \\
	= & \frac{\pi}{2k_n} \left[\int_0^1 \frac{1}{t} J_{n+\frac{3}{2}}^2\left(\left(n+\frac{3}{2}\right)t \right) \d t + \int_1^{\frac{k_n}{n+\frac{3}{2}}} \frac{1}{t} J_{n+\frac{3}{2}}^2\left(\left(n+\frac{3}{2}\right)t \right) \d t\right],
\end{align*}
where $n$ is sufficiently large. Note that
\begin{align*}
	\frac{t^{1-\frac{1}{2n+3}}\mathrm{e}^{\sqrt{1-t^2}}}{1+\sqrt{1-t^2}} < \frac{t^{\frac{1}{2}}\mathrm{e}^{\sqrt{1-t^2}}}{1+\sqrt{1-t^2}} < 1,
\end{align*}
holds for all $0<t<\frac{1}{2}$. Combining \eqref{def_bl_Jnnr} and \eqref{ineq_bl_re/1+1-r2}, we derive 
\begin{align*} 
	& \int_0^1 \frac{1}{t} J_{n+\frac{3}{2}}^2\left(\left(n+\frac{3}{2}\right)t \right) \d t  \\
	= & \frac{1}{2\pi\left(n+\frac{3}{2}\right)} \int_0^1 \frac{1}{t\sqrt{1-t^2}}\left(\frac{te^{\sqrt{1-t^2}}}{1+\sqrt{1-t^2}}\right)^{2n+3} \left(1+o\left(1\right)\right) \d t  \\
	< & \frac{1}{2\pi\left(n+\frac{3}{2}\right)} \int_0^{\frac{1}{2}} \frac{1}{\sqrt{1-t^2}}\left(\frac{t^{1-\frac{1}{2n+3}}\mathrm{e}^{\sqrt{1-t^2}}}{1+\sqrt{1-t^2}}\right)^{2n+3} \left(1+o\left(1\right)\right) \d t  \\
	& + \frac{1}{\pi\left(n+\frac{3}{2}\right)} \int_{\frac{1}{2}}^1 \frac{1}{\sqrt{1-t^2}}\left(\frac{te^{\sqrt{1-t^2}}}{1+\sqrt{1-t^2}}\right)^{2n+3} \left(1+o\left(1\right)\right) \d t  \\
	< & \frac{1}{\pi\left(n+\frac{3}{2}\right)} \int_0^1 \frac{1}{\sqrt{1-t^2}} \left(1+o\left(1\right)\right) \d t = \frac{1}{2n+3} \left(1+o\left(1\right)\right).
\end{align*}
Since $\left|J_{\nu}\left(r\right)\right|\leq1$ for all $\nu, r\in\mathbb{R}_+$, we obtain 
\begin{align*} 
	\int_1^{\frac{k_n}{n+\frac{3}{2}}} \frac{1}{t} J_{n+\frac{3}{2}}^2\left(\left(n+\frac{3}{2}\right)t \right) \d t < \frac{k_n}{n+\frac{3}{2}} - 1.
\end{align*}
Consequently, it follows that
\begin{align} \label{est_grad_j2n+1r2}
	\int_0^1 j_{n+1}^2\left(k_n r\right) r^2 \d r < \frac{\pi\left(k_n-n-1\right)}{k_n\left(2n+3\right)} \left(1+o\left(1\right)\right),
\end{align}
where $n$ is sufficiently large. Similarly, one can show that
\begin{align} \label{est_grad_j2n}
	\int_0^1 j_n^2\left(k_n r\right) \d r < \frac{\pi\left(k_n-n\right)}{k_n\left(2n+1\right)} \left(1+o\left(1\right)\right).
\end{align}
Consequently, for any fixed $m\in\mathbb{Z}$ and all large enough $n$, using \eqref{eq_bl_kn_1}, \eqref{eq_bl_r{n,si(n)}}, \eqref{eq_bl_YYgradYgradY}, \eqref{est_grad_j2n+1r2} and \eqref{est_grad_j2n}, it turns out that
\begin{align} \label{est_grad_Hnm}
	& \left\|\bm{H}_n^m\left(\bm{x}\right)\right\|^2_{L^2\left(\Omega\right)^3} \notag \\
	= & \lvert b_n^m\rvert^2 \int_0^1 \left(\frac{n^2\left(n+1\right)^2}{{k_n}^2} j_n^2\left(k_n r\right) + n\left(n+1\right)\left(\frac{n+1}{k_n r} j_n\left(k_n r\right) - j_{n+1}\left(k_n r\right)\right)^2 r^2\right) \d r \notag \\
	\leq & \lvert b_n^m\rvert^2 \int_0^1 \left(\frac{n\left(n+1\right)^2\left(3n+2\right)}{{k_n}^2} j_n^2\left(k_n r\right) + 2n\left(n+1\right) j_{n+1}^2\left(k_n r\right) r^2\right) \d r \notag \\
	< & \lvert b_n^m\rvert^2 \left[\frac{n\left(n+1\right)^2\left(3n+2\right)}{{k_n}^2} \frac{\pi\left(k_n-n\right)}{k_n\left(2n+1\right)} + 2n\left(n+1\right) \frac{\pi\left(k_n-n-1\right)}{k_n\left(2n+3\right)}\right] \left(1+o\left(1\right)\right) \notag \\
	< & \lvert b_n^m\rvert^2 2^{-2} 3^{\frac{2}{3}} \pi^{\frac{5}{3}} 5 n^{1-\xi_2} \left(1+o\left(1\right)\right).
\end{align}

\medskip

\textit{Step 2: Decomposition of $\nabla\bm{H}_n^m$ and estimates for the radial factor.}
We define 
\begin{align*}
	H_{n,\hat{\bm{r}}}^m & := \frac{n\left(n+1\right)}{k_n r} j_n\left(k_n r\right) Y^m_n\left(\theta,\phi\right), \\
	H_{n,\hat{\bm{\theta}}}^m & := \left[\frac{n+1}{k_n r}j_n\left(k_n r\right) - j_{n+1}\left(k_n r\right)\right] \frac{\partial Y^m_n\left(\theta,\phi\right)}{\partial\theta}, \\
	H_{n,\hat{\bm{\phi}}}^m & := \left[\frac{n+1}{k_n r}j_n\left(k_n r\right) - j_{n+1}\left(k_n r\right)\right] \frac{1}{\sin\theta} \frac{\partial Y^m_n\left(\theta,\phi\right)}{\partial\phi}.
\end{align*}
Then $\bm{H}_n^m$ can be rewritten as
\begin{align*}
	\bm{H}_n^m\left(\bm{x}\right) = \frac{b_n^m}{\mathrm{i}} \left(H_{n,\hat{\bm{r}}}^m\hat{\bm{r}} + H_{n,\hat{\bm{\theta}}}^m\hat{\bm{\theta}} + H_{n,\hat{\bm{\phi}}}^m\hat{\bm{\phi}}\right), 
\end{align*}
for any fixed $m\in\mathbb{Z}$ and all sufficiently large $n$. Utilizing the recurrence relations
\begin{align} \label{eq_grad_jnprimejn-1jn}
	j_n^{\prime}\left(r\right) = j_{n-1}\left(r\right) - \frac{n+1}{r} j_n\left(r\right), \quad j_n^{\prime}\left(r\right) = \frac{n}{r} j_n\left(r\right) - j_{n+1}\left(r\right),
\end{align}
we obtain the tensor expansion 
\begin{align} \label{eq_grad_gradHnm}
	& \frac{\mathrm{i}}{b_n^m} \nabla\bm{H}_n^m\left(\bm{x}\right) = H_{1,n,m} \hat{\bm{r}}\otimes\hat{\bm{r}} + H_{2,n,m} \hat{\bm{r}}\otimes\hat{\bm{\theta}} + H_{3,n,m} \hat{\bm{r}}\otimes\hat{\bm{\phi}}+ H_{4,n,m} \hat{\bm{\theta}}\otimes\hat{\bm{r}}  \notag \\
	& + H_{5,n,m} \hat{\bm{\theta}}\otimes\hat{\bm{\theta}} + H_{6,n,m} \hat{\bm{\theta}}\otimes\hat{\bm{\phi}} + H_{7,n,m} \hat{\bm{\phi}}\otimes\hat{\bm{r}} + H_{8,n,m} \hat{\bm{\phi}}\otimes\hat{\bm{\theta}} + H_{9,n,m} \hat{\bm{\phi}}\otimes\hat{\bm{\phi}},
\end{align}
where 
\begin{align}
	H_{1,n,m} := & \frac{\partial H_{n,\hat{\bm{r}}}^m}{\partial r} = \frac{n\left(n+1\right)}{k_n r^2} \left[\left(n-1\right) j_n\left(k_n r\right) - k_n r j_{n+1}\left(k_n r\right)\right] Y^m_n\left(\theta,\phi\right), \notag \\
	H_{2,n,m} := & \frac{1}{r}\frac{\partial H_{n,\hat{\bm{r}}}^m}{\partial \theta} - \frac{H_{n,\hat{\bm{\theta}}}^m}{r} = \left[\frac{n^2 - 1}{k_n r^2} j_n\left(k_n r\right) + \frac{1}{r} j_{n+1}\left(k_n r\right)\right] \frac{\partial Y^m_n\left(\theta,\phi\right)}{\partial \theta}, \notag \\
	H_{3,n,m} := & \frac{1}{r\sin\theta}\frac{\partial H_{n,\hat{\bm{r}}}^m}{\partial \phi} - \frac{H_{n,\hat{\bm{\phi}}}^m}{r} = \left[\frac{n^2 - 1}{k_n r^2} j_n\left(k_n r\right) + \frac{1}{r} j_{n+1}\left(k_n r\right)\right] \frac{1}{\sin\theta} \frac{\partial Y^m_n\left(\theta,\phi\right)}{\partial \phi}, \label{def_grad_H3nm_1} \\
	H_{4,n,m} := & \frac{\partial H_{n,\hat{\bm{\theta}}}^m}{\partial r} = \left[\left(\frac{n^2 - 1}{k_n r^2} - k_n\right) j_n\left(k_n r\right) + \frac{1}{r} j_{n+1}\left(k_n r\right)\right] \frac{\partial Y^m_n\left(\theta,\phi\right)}{\partial \theta}, \notag \\
	H_{5,n,m} := & \frac{1}{r}\frac{\partial H_{n,\hat{\bm{\theta}}}^m}{\partial \theta} + \frac{H_{n,\hat{\bm{r}}}^m}{r} = \left[\frac{n+1}{k_n r^2} j_n\left(k_n r\right) - \frac{1}{r} j_{n+1}\left(k_n r\right)\right] \frac{\partial^2 Y^m_n\left(\theta,\phi\right)}{\partial \theta^2} \notag \\
	& + \frac{n\left(n+1\right)}{k_n r^2} j_n\left(k_n r\right) Y^m_n\left(\theta,\phi\right), \notag \\
	H_{6,n,m} := & \frac{1}{r\sin\theta}\frac{\partial H_{n,\hat{\bm{\theta}}}^m}{\partial \phi} - \frac{\cos\theta}{\sin\theta} \frac{H_{n,\hat{\bm{\phi}}}^m}{r} \notag \\
	= & \left[\frac{n+1}{k_n r^2} j_n\left(k_n r\right) - \frac{1}{r} j_{n+1}\left(k_n r\right)\right] \left(\frac{1}{\sin\theta} \frac{\partial^2 Y^m_n\left(\theta,\phi\right)}{\partial \theta\partial\phi} - \frac{\cos\theta}{\sin^2\theta} \frac{\partial Y^m_n\left(\theta,\phi\right)}{\partial\phi}\right), \notag \\
	H_{7,n,m} := & \frac{\partial H_{n,\hat{\bm{\phi}}}^m}{\partial r} = \left[\left(\frac{n^2 - 1}{k_n r^2} - k_n\right) j_n\left(k_n r\right) + \frac{1}{r} j_{n+1}\left(k_n r\right)\right] \frac{1}{\sin\theta} \frac{\partial Y^m_n\left(\theta,\phi\right)}{\partial \phi}, \notag \\
	H_{8,n,m} := & \frac{1}{r} \frac{\partial H_{n,\hat{\bm{\phi}}}^m}{\partial \theta} \notag \\
	= & \left[\frac{n+1}{k_n r^2} j_n\left(k_n r\right) - \frac{1}{r} j_{n+1}\left(k_n r\right)\right] \left(- \frac{\cos\theta}{\sin^2\theta} \frac{\partial Y^m_n\left(\theta,\phi\right)}{\partial \phi} + \frac{1}{\sin\theta} \frac{\partial^2 Y^m_n\left(\theta,\phi\right)}{\partial \theta \partial \phi}\right), \notag \\
	H_{9,n,m} := & \frac{1}{r\sin\theta}\frac{\partial H_{n,\hat{\bm{\phi}}}^m}{\partial \phi} + \frac{\cos\theta}{\sin\theta} \frac{H_{n,\hat{\bm{\theta}}}^m}{r} + \frac{H_{n,\hat{\bm{r}}}^m}{r} = \frac{n\left(n+1\right)}{k_n r^2} j_n\left(k_n r\right) Y^m_n\left(\theta,\phi\right) \notag \\
	& + \left[\frac{n+1}{k_n r^2} j_n\left(k_n r\right) - \frac{1}{r} j_{n+1}\left(k_n r\right)\right] \left(\frac{\cos\theta}{\sin\theta} \frac{\partial Y^m_n\left(\theta,\phi\right)}{\partial \theta} + \frac{1}{\sin^2\theta} \frac{\partial^2 Y^m_n\left(\theta,\phi\right)}{\partial \phi^2}\right). \notag
\end{align}
Here the notation $\otimes$ is defined as $\bm{a}\otimes\bm{b} = \bm{a}\bm{b}^{\top}$ for $\bm{a}, \bm{b} \in\mathbb{C}^3$. Since $\left\{\hat{\bm{r}},\hat{\bm{\theta}},\hat{\bm{\phi}}\right\}$ form an orthonormal basis, the pointwise Frobenius norm satisfies $\|\mathbf{A}\|_F^2 = \sum_{i,j} |A_{ij}|^2$. Consequently, we deduce
\begin{align} \label{eq_grad_gradHnmLinfty}
	 & \left\|\frac{\mathrm{i}}{b_n^m} \nabla\bm{H}_n^m\left(\bm{x}\right)\right\|_{L^{\infty}\left(\Omega\setminus\overline{\Omega_{\tau}}\right)^{3\times3}}^2 
	 = \sup_{\bm{x}\in\Omega\setminus\overline{\Omega_{\tau}}} \sum_{j=1}^{9} \left|H_{j,n,m}\right|^2.
\end{align}
To derive the lower bounds, we focus on the components $H_{2,n,m}$ and $H_{3,n,m}$. Since these two components share the same radial factor, we first estimate this common part. Their angular parts will be treated separately in Steps 3 and 4.

By \cite{watson1922treatise}, one has
\begin{align*}
	J_{\nu}\left(\nu\right) = \frac{\Gamma\left(1/3\right)}{2^{2/3} 3^{1/6} \pi} \nu^{-1/3} - \frac{3^{5/6} \Gamma\left(5/3\right)}{2^{1/3} 140 \pi} \nu^{-5/3} + o\left(\nu^{-5/3}\right), \quad \nu\to\infty.
\end{align*}
We define
\begin{align} \label{def_grad_rn}
	r_n := \frac{n+\frac{1}{2}}{k_n} = \frac{1}{1+c\left(n+\frac{1}{2}\right)^{-\xi}},
\end{align} 
where $0<\xi<\frac{2}{3}$ and $c>0$. In particular, $k_n r_n = n+\frac{1}{2}$ and $r_n\to 1^-$ as $n\to\infty$. It follows that
\begin{align} \label{eq_grad_jnknrn} 
	j_n\left(k_n r_n\right) = \sqrt{\frac{\pi}{2n+1}} J_{n+\frac{1}{2}}\left(n+\frac{1}{2}\right) = c_3 n^{-5/6} + O\left(n^{-11/6}\right),
\end{align}
where we set $c_3 := \frac{\Gamma\left(1/3\right)}{2^{7/6} 3^{1/6} \sqrt{\pi}}$. Since $\left|J_{\nu}\left(r\right)\right| \leq 1$ for all $\nu,r\in\mathbb{R_+}$, we have
\begin{align*}
	\left|j_{n+1}\left(n+\frac{1}{2}\right)\right| 
	\leq \sqrt{\frac{\pi}{2n+1}}.
\end{align*}
Hence one can derive that
\begin{align} \label{est_grad_jn+jn+1}
	& \sup_{r\in\left(\tau,1\right)} \left|\frac{n^2 - 1}{k_n r^2} j_n\left(k_n r\right) + \frac{1}{r} j_{n+1}\left(k_n r\right)\right| \notag \\
	\geq & \frac{1}{r_n} \left(\left| \frac{n^2 - 1}{k_n r_n} j_n\left(k_n r_n\right)\right| - \left|j_{n+1}\left(k_n r_n\right) \right|\right) \notag \\
	\geq & \left(1+O\left(n^{-\xi}\right)\right) \left[\frac{n^2 - 1}{n+\frac{1}{2}} \left(c_3 n^{-5/6} + O\left(n^{-11/6}\right)\right)  - \sqrt{\frac{\pi}{2n+1}}\right] \notag \\
	= & c_3 n^{1/6} + O\left(n^{1/6-\xi}\right). 
\end{align}

\medskip

\textit{Step 3: Lower bounds for $\nabla\bm{H}_n^m$: the case $m=0$.}
In this case, $H_{2,n,0}$ takes the form
\begin{align} \label{eq_grad_H2n0}
	H_{2,n,0} = - \sqrt{\frac{2n+1}{4\pi}} \left[\frac{n^2 - 1}{k_n r^2} j_n\left(k_n r\right) + \frac{1}{r} j_{n+1}\left(k_n r\right)\right] \sin\theta P^{\prime}_n\left(\cos\theta\right).
\end{align}
By \cite{szeg1939orthogonal}, the Legendre polynomials satisfy the recurrence relation
\begin{align*} 
	\sin^2 \theta P^{\prime}_n\left(\cos\theta\right) = \left(n+1\right) \cos\theta P_n\left(\cos\theta\right) - \left(n+1\right) P_{n+1}\left(\cos\theta\right),
\end{align*}
and the asymptotic expansion
\begin{align*} 
	P_n\left(\cos\theta\right) = \sqrt{\frac{2}{\pi n \sin\theta}} \cos\left(\alpha_n\left(\theta\right)\right) + O\left(n^{-\frac{3}{2}}\right), \quad \varepsilon<\theta<\pi-\varepsilon,
\end{align*}
with $\alpha_n\left(\theta\right) := \left(n+\frac{1}{2}\right)\theta - \frac{\pi}{4}$ and $0<\varepsilon<\frac{\pi}{2}$. 
We can derive
\begin{align} \label{est_grad_sinPnprime}
	& \sup_{\theta\in\left[0,\pi\right]} \left|\sin\theta P^{\prime}_n\left(\cos\theta\right)\right| 
	\notag \\
	\geq & \sup_{\theta\in\left[\varepsilon,\pi-\varepsilon\right]} \left|\left(n+1\right) \cos\theta P_n\left(\cos\theta\right) - \left(n+1\right) P_{n+1}\left(\cos\theta\right)\right| \notag \\
	= & \sup_{\theta\in\left[\varepsilon,\pi-\varepsilon\right]} \left|\sqrt{\frac{2\left(n+1\right)}{\pi}} \sqrt{\sin\theta}\sin\left(\alpha_n\left(\theta\right)\right) + O\left(n^{-1/2}\right)\right| \notag \\
	\geq & \left|\sqrt{\frac{2\left(n+1\right)}{\pi}} \sqrt{\sin\frac{\pi}{3}} \sin\left(\alpha_n\left(\frac{\pi}{3}\right)\right) + O\left(n^{-1/2}\right)\right| \notag \\
	\geq & c_4 \sqrt{n} + O\left(n^{-1/2}\right),
\end{align}
where $c_4 := \sqrt{\frac{2}{\pi}} \sqrt{\sin\frac{\pi}{3}} \sin\frac{\pi}{12}$. Combining \eqref{est_grad_jn+jn+1}, \eqref{eq_grad_H2n0} and \eqref{est_grad_sinPnprime}, we obtain
\begin{align} \label{est_grad_H20}
	& \sup_{\bm{x}\in\Omega\setminus\overline{\Omega_{\tau}}} \left|H_{2,n,0}\right| 
	\notag \\
	= & \sqrt{\frac{2n+1}{4\pi}} \sup_{r\in\left(\tau,1\right)} \left| \frac{n^2 - 1}{k_n r^2} j_n\left(k_n r\right) + \frac{1}{r} j_{n+1}\left(k_n r\right) \right| \sup_{\theta\in\left[0,\pi\right]} \left| \sin\theta P^{\prime}_n\left(\cos\theta\right) \right| \notag \\
	\geq & \sqrt{\frac{2n+1}{4\pi}} \left(c_3 n^{\frac{1}{6}} + O\left(n^{\frac{1}{6}-\xi}\right)\right) \left(c_4 \sqrt{n} + O\left(n^{-\frac{1}{2}}\right)\right) \notag \\
	= & c_5 n^{\frac{7}{6}} + O\left(n^{\frac{1}{6}}\right),
\end{align}
where $c_5 := \frac{\Gamma\left(\frac{1}{3}\right) \left(\sqrt{3} - 1\right) 3^{\frac{1}{12}}}{2^{\frac{19}{6}} \pi^{\frac{3}{2}}}$. Therefore, by \eqref{eq_grad_gradHnmLinfty} and \eqref{est_grad_H20}, it follows that 
\begin{align*}
	\left\|\nabla\bm{H}_n^0\left(\bm{x}\right)\right\|_{L^{\infty}\left(\Omega\setminus\overline{\Omega_{\tau}}\right)^{3\times3}} \geq \left|b_n^0\right| \sup_{\bm{x}\in\Omega\setminus\overline{\Omega_{\tau}}} \left|H_{2,n,0}\right| \geq \left|b_n^0\right| \left(c_5 n^{\frac{7}{6}} + O\left(n^{\frac{1}{6}}\right)\right).
\end{align*}
Combining this with \eqref{est_grad_Hnm}, we obtain \eqref{est_grad_gradHn0}.

\medskip

\textit{Step 4: Lower bounds for $\nabla\bm{H}_n^m$: the case $m\neq0$.}
Fix $m\in\mathbb{Z}\setminus\{0\}$. We now establish the estimate \eqref{est_grad_gradHnm} for sufficiently large $n$. Since $\lvert m\rvert\leq n$ automatically holds when $n$ is sufficiently large, it suffices to estimate the $L^{\infty}$-norm of $H_{3,n,m}$. 

From \eqref{def_bl_Ynm} and \eqref{def_grad_H3nm_1}, we have 
\begin{align} \label{def_grad_H3nm_2}
	H_{3,n,m} = \left(\frac{n^2 - 1}{k_n r^2} j_n\left(k_n r\right) + \frac{1}{r} j_{n+1}\left(k_n r\right)\right) \frac{\mathrm{i}m c_{n,m}}{\sin\theta} P_n^{\left|m\right|}\left(\cos\theta\right) e^{\mathrm{i}m\phi}, 
\end{align}
where $c_{n,m}$ is defined by \eqref{def_bl_cnm}. By \cite{lohofer1998inequalities}, for each $n\geq1$ and $1\leq \left|m\right| \leq n$, the function $\left|P_n^{\left|m\right|}\left(\cos\theta\right)\right|$ attains its maximum at some point $\theta_{n,\left|m\right|}\in\left(0,\pi/2\right]$, that is,
\begin{align} \label{def_grad_thetanm}
	\left|P_n^{\left|m\right|}\left(\cos\theta_{n,\left|m\right|}\right)\right| = \sup_{\theta\in\left[0,\pi\right]} \left|P_n^{\left|m\right|}\left(\cos\theta\right)\right|.
\end{align} 
Moreover, it holds that 
\begin{align} \label{ineq_grad_sinthetanm}
	\frac{\left|m\right|}{n+\frac{1}{2}} \leq \sin\theta_{n,\left|m\right|} \leq \frac{1.11 \left(\left|m\right| + 1\right)}{n+\frac{1}{2}}.
\end{align}
In addition, \cite{lohofer1998inequalities} shows that $P_n^{\left|m\right|}(x)$ satisfies 
\begin{align} \label{est_grad_Pnm}
	\frac{1}{\sqrt{2.22\left(\left|m\right|+1\right)}} < \sup_{x\in\left[-1,1\right]} \left|P_n^{\left|m\right|}(x)\right| \sqrt{\frac{\left(n-\left|m\right|\right)!}{\left(n+\left|m\right|\right)!}} < \frac{2^{\frac{5}{4}}}{\pi^{\frac{3}{4}}} \frac{1}{\left|m\right|^{\frac{1}{4}}},
\end{align}
for all $n\geq 1$ and $1\leq \left|m\right| \leq n$. By virtue of \eqref{def_grad_thetanm}, \eqref{ineq_grad_sinthetanm} and \eqref{est_grad_Pnm}, for sufficiently large $n$, we have
\begin{align} \label{est_grad_1/sincPem}
	\sup_{\theta\in\left[0,\pi\right], \phi\in\left[0,2\pi\right)} \left|\frac{\mathrm{i}m c_{n,m}}{\sin\theta}  P_n^{\left|m\right|}\left(\cos\theta\right) e^{\mathrm{i}m\phi} \right| \geq & \left|\frac{m c_{n,m}}{\sin\theta_{n,\left|m\right|}}  P_n^{\left|m\right|}\left(\cos\theta_{n,\left|m\right|}\right)\right| \notag \\
	\geq & \frac{1}{2\sqrt{\pi}} \frac{\left(n+\frac{1}{2}\right)^{3/2}}{\left(1.11\left(\left|m\right| + 1\right)\right)^{3/2}} \left|m\right|.
\end{align}
Since $\lvert m \rvert\geq1$, we have $\lvert m \rvert+1\leq2\lvert m \rvert$, and therefore
\begin{align*}
    \frac{\left|m\right|}{\left(\left|m\right| + 1\right)^{3/2} }\geq 2^{-3/2}\left|m\right|^{-1/2}.
\end{align*}
Using the estimates \eqref{est_grad_jn+jn+1} and \eqref{est_grad_1/sincPem}, we deduce
\begin{align} \label{est_grad_H3Linfty}
	& \sup_{\bm{x}\in\Omega\setminus\overline{\Omega_{\tau}}} \left|H_{3,n,m}\right| \notag \\
	= & \sup_{r\in\left(\tau,1\right)} \left| \frac{n^2 - 1}{k_n r^2} j_n\left(k_n r\right) + \frac{1}{r} j_{n+1}\left(k_n r\right) \right| \sup_{\theta\in\left[0,\pi\right], \phi\in\left[0,2\pi\right]} \left|\frac{\mathrm{i}m c_{n,m}}{\sin\theta}  P_n^{\left|m\right|}\left(\cos\theta\right) e^{\mathrm{i}m\phi} \right| \notag \\
	\geq & \left(c_3 n^{1/6} + O\left(n^{1/6-\xi}\right)\right) \frac{1}{2\sqrt{\pi}} \frac{\left(n+\frac{1}{2}\right)^{3/2}}{\left(1.11\left(\left|m\right| + 1\right)\right)^{3/2}} \left|m\right| \notag \\
	\geq & c_6 n^{5/3} \left|m\right|^{-1/2} + O\left(n^{5/3-\xi} \left|m\right|^{-1/2}\right), 
\end{align}
where $c_6 := \frac{\Gamma\left(1/3\right)}{\left(1.11\right)^{3/2} 2^{11/3} 3^{1/6} \pi}$. Therefore, by \eqref{eq_grad_gradHnmLinfty} and \eqref{est_grad_H3Linfty}, we obtain 
\begin{align*}
	\left\|\nabla\bm{H}_n^m\left(\bm{x}\right)\right\|_{L^{\infty}\left(\Omega\setminus\overline{\Omega_{\tau}}\right)^{3\times3}} \geq \left|b_n^m\right| \sup_{\bm{x}\in\Omega\setminus\overline{\Omega_{\tau}}} \left|H_{3,n,m}\right| \geq \left|b_n^m\right| \left(c_6 \frac{n^{5/3}}{\sqrt{\left|m\right|}} + O\left(\frac{n^{5/3-\xi}}{\sqrt{\left|m\right|}}\right)\right).
\end{align*}
Combining this with \eqref{est_grad_Hnm}, we obtain \eqref{est_grad_gradHnm}.

\medskip

\textit{Step 5: Estimates for $\bm{E}_n^m$.} Fix $m\in\mathbb{Z}$. For sufficiently large $n$, the function $\bm{E}_n^m$ is given by \eqref{def_bl_En}. Using \eqref{est_grad_j2n}, we obtain 
\begin{align} \label{est_grad_Enm}
	\left\|\bm{E}_n^m\left(\bm{x}\right)\right\|^2_{L^2\left(\Omega\right)^3} & = \lvert b_n^m\rvert^2 n\left(n+1\right) \int_0^1 j_n^2\left(k_n r\right) \d r \notag \\
	& < \lvert b_n^m\rvert^2 2^{-2} 3^{2/3} \pi^{5/3} n^{1-\xi_2} \left(1+o\left(1\right)\right).
\end{align}
The analysis for $\nabla\bm{E}_n^m$ follows the same lines as that for $\nabla\bm{H}_n^m$. We write $\bm{E}_n^m$ as
\begin{align*}
	\bm{E}_n^m\left(\bm{x}\right) = b_n^m \left(E_{n,\hat{\bm{\theta}}}^m\hat{\bm{\theta}} + E_{n,\hat{\bm{\phi}}}^m\hat{\bm{\phi}}\right),
\end{align*}
where we define
\begin{align*}
	E_{n,\hat{\bm{\theta}}}^m := j_n\left(k_n r\right) \frac{1}{\sin\theta} \frac{\partial Y^m_n\left(\theta,\phi\right)}{\partial\phi}, \quad E_{n,\hat{\bm{\phi}}}^m := - j_n\left(k_n r\right) \frac{\partial Y^m_n\left(\theta,\phi\right)}{\partial\theta}.
\end{align*}
Using \eqref{eq_grad_jnprimejn-1jn}, we obtain 
\begin{align} \label{eq_grad_gradEnm}
	\frac{1}{b_n^m} \nabla\bm{E}_n^m\left(\bm{x}\right) = & E_{1,n,m} \hat{\bm{r}}\otimes\hat{\bm{\theta}} + E_{2,n,m} \hat{\bm{r}}\otimes\hat{\bm{\phi}}+ E_{3,n,m} \hat{\bm{\theta}}\otimes\hat{\bm{r}} + E_{4,n,m} \hat{\bm{\theta}}\otimes\hat{\bm{\theta}} \notag \\
	& + E_{5,n,m} \hat{\bm{\theta}}\otimes\hat{\bm{\phi}} + E_{6,n,m} \hat{\bm{\phi}}\otimes\hat{\bm{r}} + E_{7,n,m} \hat{\bm{\phi}}\otimes\hat{\bm{\theta}} + E_{8,n,m} \hat{\bm{\phi}}\otimes\hat{\bm{\phi}},
\end{align}
where 
\begin{align}
	E_{1,n,m} := & - \frac{E_{n,\hat{\bm{\theta}}}^m}{r} = - \frac{j_n\left(k_n r\right)}{r} \frac{1}{\sin\theta} \frac{\partial Y^m_n\left(\theta,\phi\right)}{\partial \phi}, \notag \\
	E_{2,n,m} := & - \frac{E_{n,\hat{\bm{\phi}}}^m}{r} = \frac{j_n\left(k_n r\right)}{r} \frac{\partial Y^m_n\left(\theta,\phi\right)}{\partial \theta}, \notag \\
	E_{3,n,m} := & \frac{\partial E_{n,\hat{\bm{\theta}}}^m}{\partial r} = \left[\frac{n}{r} j_n\left(k_n r\right) - k_n j_{n+1}\left(k_n r\right) \right] \frac{1}{\sin\theta} \frac{\partial Y^m_n\left(\theta,\phi\right)}{\partial \phi}, \notag \\
	E_{4,n,m} := & \frac{1}{r}\frac{\partial E_{n,\hat{\bm{\theta}}}^m}{\partial \theta} = - \frac{j_n\left(k_n r\right)}{r} \frac{\cos\theta}{\sin^2\theta} \frac{\partial Y^m_n\left(\theta,\phi\right)}{\partial\phi}, \notag \\
	E_{5,n,m} := & \frac{1}{r\sin\theta}\frac{\partial E_{n,\hat{\bm{\theta}}}^m}{\partial \phi} - \frac{\cos\theta}{\sin\theta} \frac{E_{n,\hat{\bm{\phi}}}^m}{r} \notag \\
	= & \frac{j_n\left(k_n r\right)}{r} \left(\frac{\cos\theta}{\sin\theta} \frac{\partial Y^m_n\left(\theta,\phi\right)}{\partial \theta} + \frac{1}{\sin^2\theta} \frac{\partial^2 Y^m_n\left(\theta,\phi\right)}{\partial \phi^2}\right), \notag \\
	E_{6,n,m} := & \frac{\partial E_{n,\hat{\bm{\phi}}}^m}{\partial r} = \left[k_n j_{n+1}\left(k_n r\right) - \frac{n}{r} j_n\left(k_n r\right) \right] \frac{\partial Y^m_n\left(\theta,\phi\right)}{\partial \theta}, \notag \\
	E_{7,n,m} := & \frac{1}{r} \frac{\partial E_{n,\hat{\bm{\phi}}}^m}{\partial \theta} = - \frac{j_n\left(k_n r\right)}{r} \frac{\partial^2 Y^m_n\left(\theta,\phi\right)}{\partial \theta^2}, \notag \\
	E_{8,n,m} := & \frac{1}{r\sin\theta}\frac{\partial E_{n,\hat{\bm{\phi}}}^m}{\partial \phi} + \frac{\cos\theta}{\sin\theta} \frac{E_{n,\hat{\bm{\theta}}}^m}{r} \notag \\
	= & \frac{j_n\left(k_n r\right)}{r} \left(\frac{\cos\theta}{\sin^2\theta} \frac{\partial Y^m_n\left(\theta,\phi\right)}{\partial \phi} - \frac{1}{\sin\theta} \frac{\partial^2 Y^m_n\left(\theta,\phi\right)}{\partial \theta \partial \phi}\right). \notag
\end{align}
By \eqref{eq_bl_sinYnmnn+1}, we have
\begin{align} \label{eq_grad_gradEnmLinfty}
	& \left\|\frac{\mathrm{i}}{b_n^m} \nabla\bm{E}_n^m\left(\bm{x}\right)\right\|_{L^{\infty}\left(\Omega\setminus\overline{\Omega_{\tau}}\right)^{3\times3}}^2 \notag \\
	\geq & \sup_{\bm{x}\in\Omega\setminus\overline{\Omega_{\tau}}} \frac{1}{2} \left|E_{5,n,m} - E_{7,n,m}\right|^2 \notag \\
	= & \sup_{\bm{x}\in\Omega\setminus\overline{\Omega_{\tau}}} \frac{1}{2} \left|\frac{j_n\left(k_n r\right)}{r}\right|^2 \left|\frac{\cos\theta}{\sin\theta} \frac{\partial Y^m_n\left(\theta,\phi\right)}{\partial \theta} + \frac{1}{\sin^2\theta} \frac{\partial^2 Y^m_n\left(\theta,\phi\right)}{\partial \phi^2} + \frac{\partial^2 Y^m_n\left(\theta,\phi\right)}{\partial \theta^2}\right|^2 \notag \\
	= & \sup_{\bm{x}\in\Omega\setminus\overline{\Omega_{\tau}}} \frac{1}{2} n^2\left(n+1\right)^2 \left|\frac{j_n\left(k_n r\right)}{r}\right|^2 \left|Y^m_n\left(\theta,\phi\right)\right|^2.
\end{align}
For $r_n$ given by \eqref{def_grad_rn}, the estimate \eqref{eq_grad_jnknrn} yields 
\begin{align} \label{est_grad_jn/r}
	\sup_{r\in\left(\tau,1\right)} \left|\frac{j_n\left(k_n r\right)}{r}\right| \geq \left|j_n\left(k_n r_n\right)\right| = c_3 n^{-5/6} + O\left(n^{-11/6}\right).
\end{align}
If $m=0$, then $P_n\left(1\right)=1$, and hence
\begin{align*}
	\sup_{\theta\in\left[0,\pi\right], \phi\in\left[0,2\pi\right]} \left|Y^0_n\left(\theta,\phi\right)\right| \geq \sqrt{\frac{2n+1}{4\pi}}.
\end{align*} 
Therefore we obtain
\begin{align*}
	\left\|\frac{1}{b_n^m} \nabla\bm{E}_n^0\left(\bm{x}\right)\right\|_{L^{\infty}\left(\Omega\setminus\overline{\Omega_{\tau}}\right)^{3\times3}} \geq & \frac{n\left(n+1\right)}{\sqrt{2}} \left(c_3 n^{-5/6} + O\left(n^{-11/6}\right)\right) \sqrt{\frac{2n+1}{4\pi}} \\
	\geq & \frac{\Gamma\left(\frac{1}{3}\right)}{2^{13/6} 3^{1/6} \pi} n^{5/3} + O\left(n^{2/3}\right).
\end{align*}
Combining this with \eqref{est_grad_Enm}, we obtain \eqref{est_grad_gradEn0}.

If $m\neq0$, then for all sufficiently large $n$ one has $\lvert m \rvert \leq n$. Using \eqref{est_grad_Pnm}, \eqref{eq_grad_gradEnmLinfty} and \eqref{est_grad_jn/r}, we deduce
\begin{align*}
	\left\|\frac{1}{b_n^m} \nabla\bm{E}_n^m\left(\bm{x}\right)\right\|_{L^{\infty}\left(\Omega\setminus\overline{\Omega_{\tau}}\right)^{3\times3}} \geq & \frac{n\left(n+1\right)}{\sqrt{2}} \left(c_3 n^{-5/6} + O\left(n^{-11/6}\right)\right) \sqrt{\frac{2n+1}{8.88\pi\left(\left|m\right|+1\right)}} \\
	\geq & \frac{\Gamma\left(\frac{1}{3}\right)}{\left(1.11\right)^{1/2} 2^{19/6} 3^{1/6} \pi} \frac{n^{5/3}}{\left|m\right|} + O\left(\frac{n^{2/3}}{\left|m\right|}\right).
\end{align*}
Combining this with \eqref{est_grad_Enm}, we obtain \eqref{est_grad_gradEnm}. 

The proof is complete.
\end{proof}

\begin{remark}
We make several remarks on the estimates in Theorem \ref{thm4.1}. First, recall that $r_n \to 1^-$ as $n \to \infty$. Hence for any fixed $\tau\in\left(0,1\right)$, one has $r_n\in\left(\tau,1\right)$ for all sufficiently large $n$. Consequently, the leading constants in the lower bounds \eqref{est_grad_gradEn0}--\eqref{est_grad_gradHnm} are independent of the choice of $\tau$.
	
Second, the asymptotic growth rate of $\bm{E}_n^0$ in \eqref{est_grad_gradEn0} is higher than that of $\bm{H}_n^0$ in \eqref{est_grad_gradHn0}, whereas the growth rates are of the same order in the case $m\neq0$. This difference stems from the fact that the estimate for $\bm{E}_n^0$ exploits \eqref{eq_bl_sinYnmnn+1}, which captures the factor $n(n+1)$ arising from the spherical Laplacian. If one applies a similar argument to $\bm{H}_n^0$, through the components $H_{5,n,m}$ and $H_{9,n,m}$, the resulting radial factor becomes
\begin{align*}
	R_n(r) := \frac{n-1}{k_n r^2} j_n\left(k_n r\right) - \frac{1}{r} j_{n+1}\left(k_n r\right).
\end{align*}
Since it is nontrivial to derive a sharp lower bound for $\sup_{r\in(\tau,1)}|R_n(r)|$, we instead use a different estimation method for $\bm{H}_n^0$, which leads to a weaker growth rate compared to $\bm{E}_n^0$. 
	
Finally, concerning the choice of components, our proof relies on $H_{2,n,0}$ in the case $m=0$ and on $H_{3,n,m}$ in the case $m\neq0$. Although some other components have the same radial asymptotic order at $r_n$, they are not used because their angular parts are either more technically involved to estimate by the present method, or yield strictly weaker lower bounds than those obtained from $H_{2,n,0}$ and $H_{3,n,m}$. These two components are therefore sufficient to establish the stated blow-up rates. We also remark that sharper $L^2(\Omega)$-norm estimates would yield sharper lower bounds.
\end{remark}

\section{Conclusions and future work} \label{section_5}

In this paper, we have investigated the EEITEP \eqref{eq_intro_EHuOmega} arising from the non-scattering phenomenon of an elastic body under electromagnetic probing. For a general bounded Lipschitz domain, we have proved that the set of positive transmission eigenvalues, if non-empty, is discrete, with $\infty$ as its only possible accumulation point. Since every positive non-scattering wavenumber of the elastic-electromagnetic scattering problems \eqref{eq_intro_EE_R3} and \eqref{eq_intro_SMRC} corresponds to a positive transmission eigenvalue of the EEITEP \eqref{eq_intro_EHuOmega}, while the converse implication does not necessarily hold, the discreteness result obtained for transmission eigenvalues applies to a larger spectral set and therefore implies the discreteness of positive non-scattering wavenumbers. Consequently, the positive non-scattering wavenumbers in the elastic-electromagnetic scattering problems can only form a discrete set, with $\infty$ as their only possible accumulation point. In the radially symmetric domain, we have further established the existence of a sequence of transmission eigenvalues and derived their asymptotic behavior as they tend to infinity. The corresponding electromagnetic components are shown to be localized near the boundary and to exhibit quantitative gradient blow-up, whereas the elastic displacement field remains globally distributed throughout the domain. These results reveal the boundary localization and high oscillation structure of the electromagnetic transmission eigenfunctions that is relevant to invisibility and inverse shape reconstruction in elastic-electromagnetic scattering. 

Although the analysis of boundary localization in the present work is restricted to the radially symmetric setting, analogous boundary-localized transmission eigenfunctions have been observed in extensive numerical experiments for a variety of non-radial scatterers in the scalar acoustic model \cite{chow2021surface}. In light of this numerical evidence and the localization mechanism identified in the present analysis, we believe that the boundary localization phenomenon of the electromagnetic components in the EEITEP \eqref{eq_intro_EHuOmega} persists for general elastic scatterers and is not an artifact of spherical symmetry. The rigorous analysis and numerical verification will be pursued in future work.

The boundary localization of electromagnetic transmission eigenfunctions motivates further investigation of the inverse problem of reconstructing the support of an unknown elastic scatterer from multi-frequency electromagnetic far-field measurements. In particular, these boundary-localized electromagnetic components suggest a potential spectral mechanism for developing super-resolution imaging methods in elastic-electromagnetic scattering. More precisely, for the scattering problem \eqref{eq_intro_EE_R3} and \eqref{eq_intro_SMRC} with a prescribed wavenumber interval $I\subset\mathbb{R}_{+}$, consider the following geometrical inverse shape problem: 
\begin{align} \label{eq:IP}
    (\bm{E}^{\infty}(\hat{\bm{x}}, \bm{d}, \bm{p};k), \bm{H}^{\infty}(\hat{\bm{x}}, \bm{d}, \bm{p};k))\mapsto \partial \Omega, \quad \hat{\bm{x}},\bm{d}\in\mathbb{S}^{2},\ \bm{p}\in\mathbb{R}^{3},\ k\in I,
\end{align}
where $\bm{E}^{\infty}$ and $\bm{H}^{\infty}$ are defined in \eqref{eq_Einfty} and \eqref{eq_Hinfty}, and the reconstruction does not require a priori knowledge of the elastic parameters $\lambda$, $\mu$, and $\rho$. The spectral results obtained in the present work suggest two complementary directions for the inverse problem \eqref{eq:IP}. 

First, the discreteness of the positive transmission eigenvalues makes it possible, in principle, to select probing wavenumbers away from the transmission spectrum. If the classical linear sampling method (LSM) framework can be established for the present elastic-electromagnetic scattering problem \eqref{eq_intro_EE_R3} and \eqref{eq_intro_SMRC}, then measurements taken at wavenumbers outside the transmission spectrum can be used to obtain a conventional reconstruction of the support of the elastic scatterer.

Second, rather than avoiding transmission eigenvalues, we propose to exploit the transmission eigenvalues together with their associated boundary-localized electromagnetic eigenfunctions, in order to investigate a potential super-resolution imaging mechanism. A related imaging scheme has been developed for time-harmonic electromagnetic medium scattering in \cite{he2024invisibility}, where Maxwell transmission eigenvalues are detected from electric far-field data, the corresponding transmission eigenfunctions are approximated by Maxwell-Herglotz waves, and an imaging functional is constructed from their geometric structures. In the present setting, however, the measured electromagnetic far-field patterns are generated by a coupled elastic-electromagnetic scattering system \eqref{eq_intro_EE_R3} and \eqref{eq_intro_SMRC}, and the interior transmission eigenfunctions are governed by the EEITEP \eqref{eq_intro_EHuOmega}. Hence, extending the above electromagnetic imaging strategy to the recovery of elastic scatterers requires new analytical and numerical investigations. The corresponding reconstruction procedure can be organized into the following three phases.

\medskip

\textit{Phase I: Identification of candidate transmission eigenvalues from far-field data.}                  

Let $I\subset \mathbb{R}_{+}$ be a prescribed interval of probing wavenumbers. By the electromagnetic far-field data in \eqref{eq:IP}, we seek to identify the elastic-electromagnetic interior transmission eigenvalues contained in $I$ based on the mechanism of LSM \cite{haddar2002linear}. 

Define 
\begin{align*}
    L_t^{2}(\mathbb{S}^{2}):=\left\{\bm{g}\in L^{2}(\mathbb{S}^{2})^{3};\,\bm{g}(\bm{d})\cdot \bm{d}=0\text{ for } \bm{d}\in\mathbb{S}^{2}\right\}.
\end{align*}
For $\bm{g}\in L_t^{2}(\mathbb{S}^{2})$ and $k\in I$, let the electric far-field operator $\mathcal{F}_k^E:L_t^{2}(\mathbb{S}^{2})\longrightarrow L_t^{2}(\mathbb{S}^{2})$
be defined as
\begin{align*}
    (\mathcal{F}_k^E\bm{g})(\hat{\bm{x}}):=\int_{\mathbb{S}^{2}}\bm{E}^{\infty}\bigl(\hat{\bm{x}},\bm{d},\bm{g}(\bm{d});k\bigr)\,\mathrm{d}s(\bm{d}),\qquad\hat{\bm{x}}\in\mathbb{S}^{2}.
\end{align*}
Associated with $\bm{g}_k\in L_t^{2}(\mathbb{S}^{2})$, the electromagnetic Herglotz pair is given by
\begin{align*}
    &\bm{E}_{\bm{g}_k,k}^{\mathrm{in}}(\bm{x})=\mathrm{i}k\int_{\mathbb{S}^{2}}e^{\mathrm{i}k\bm{x}\cdot\bm{d}}\bm{g}_k(\bm{d})\,\mathrm{d}s(\bm{d}),\qquad\bm{x}\in\mathbb{R}^{3},\\
    &\bm{H}_{\bm{g}_k,k}^{\mathrm{in}}(\bm{x})=\frac{1}{\mathrm{i}k}\mathbf{curl}\ \bm{E}_{\bm{g}_k,k}^{\mathrm{in}}(\bm{x})=\mathrm{i}k\int_{\mathbb{S}^{2}}e^{\mathrm{i}k\bm{x}\cdot\bm{d}}\bigl(\bm{d}\times\bm{g}_k(\bm{d})\bigr)\,\mathrm{d}s(\bm{d}),\qquad\bm{x}\in\mathbb{R}^{3}.
\end{align*}
Then $\mathcal{F}_k^E\bm{g}_k$ is electric far-field pattern of the scattered field generated by $\left(\bm{E}_{\bm{g},k}^{\mathrm{in}},\bm{H}_{\bm{g},k}^{\mathrm{in}}\right)$ as the incident field. For a polarization vector $\bm{q}\in\mathbb{S}^2$, let
\begin{align*}
    \bm{E}_{e,k}^{\infty}(\hat{\bm{x}},\bm{z},\bm{q}):=\frac{\mathrm{i}k}{4\pi}(\hat{\bm{x}} \times \bm{q})e^{-\mathrm{i}k\hat{\bm{x}} \cdot \bm{z}},
\end{align*}
denote the electric far-field pattern of an electric dipole located at $\bm{z}$. Assuming that an interior sampling point $\bm z_0\in\Omega$ is available, one may consider the far-field equation 
\begin{align} \label{eq_ffeq}
    \left(\mathcal{F}_k^E \bm{g}_{k}(\cdot,\bm{z}_0,\bm{q})\right)(\hat{\bm{x}})=\bm{E}_{e,k}^{\infty}(\hat{\bm{x}},\bm{z}_0,\bm{q}).
\end{align}
For the time-harmonic electromagnetic medium scattering, it shows that the $L^2$-norm of the associated electric Maxwell-Herglotz wave function generated by a regularized solution of the corresponding far-field equation remains bounded for almost every interior sampling point, as the regularization parameter tends to zero, if and only if $k$ is not a Maxwell transmission eigenvalue \cite{he2024invisibility}. This result provides the theoretical basis for detecting Maxwell transmission eigenvalues by monitoring the growth of the corresponding Herglotz kernels. Motivated by this mechanism, for the elastic-electromagnetic scattering problem \eqref{eq_intro_EE_R3} and \eqref{eq_intro_SMRC}, one may compute a suitable regularized solution $\bm{g}_{k,\varepsilon}(\cdot,\bm{z}_0,\bm{q})$ of \eqref{eq_ffeq} and examine the behavior of $\left\| \bm{g}_{k,\varepsilon}(\cdot,\bm{z}_0,\bm{q}) \right\|_{L_t^{2}(\mathbb{S}^2)}$ as a function of $k\in I$. If an analogous LSM characterization can be established for the scattering problem \eqref{eq_intro_EE_R3} and \eqref{eq_intro_SMRC}, namely, if $\left\| \bm{E}_{\bm{g}_{k,\varepsilon},k}^{\mathrm{in}}(\bm{x})\right\|_{L^2(\Omega)^3}$ remains bounded as $\varepsilon\to0$ when $k$ is not an elastic-electromagnetic transmission eigenvalue and fails to remain bounded as $\varepsilon\to0$ when $k$ is an elastic-electromagnetic transmission eigenvalue, then the growth of $\left\| \bm{g}_{k,\varepsilon}(\cdot,\bm{z}_0,\bm{q}) \right\|_{L_t^{2}(\mathbb{S}^2)}$ as a function of $k\in I$ may serve as indicators of candidate elastic-electromagnetic transmission eigenvalues. The rigorous establishment of this characterization is left for future work.

\medskip 

\textit{Phase II: Approximation of boundary-localized electromagnetic transmission eigenfunctions.}  

Let $\Lambda_L=\{k_1,k_2,\ldots,k_L\}$ be the collection of candidate transmission eigenvalues identified in Phase I. In the time-harmonic electromagnetic medium scattering, it has been shown that, when $k$ is a Maxwell transmission eigenvalue, there exists a Maxwell-Herglotz incident field whose induced scattered far-field pattern is nearly vanishing and whose restriction to the scatterer approximates a corresponding Maxwell transmission eigenfunction \cite{he2024invisibility}. Motivated by this mechanism, we propose to seek analogous nearly non-scattering incident electromagnetic fields for the elastic-electromagnetic scattering problem \eqref{eq_intro_EE_R3} and \eqref{eq_intro_SMRC}.

Let $D$ be a known bounded domain containing $\Omega$. For each $k_l\in\Lambda_L$, one may seek the Herglotz kernel
$\bm g_{k_l}\in L_t^{2}(\mathbb{S}^{2})$ by considering the optimization problem
\begin{align*}
    \min_{\bm{g}\in L_t^{2}(\mathbb{S}^{2})}\left\|\mathcal{F}_{k_l}^E\bm{g}\right\|_{L_t^{2}(\mathbb{S}^{2})}\quad\text{subject to}\quad\left\|\bm E_{\bm{g},k_l}^{\mathrm{in}}\right\|_{H(\mathbf{curl}, D)}=1.
\end{align*}
The minimization of $\|\mathcal{F}_{k_l}^{E}\bm g\|_{L_t^{2}(\mathbb{S}^{2})}$ is intended to produce an incident electromagnetic field whose associated scattered electric far-field pattern is nearly vanishing, while the normalization constraint excludes the trivial field. If an analogue of the Maxwell-Herglotz approximation result can be established, then a suitable regularized approximate minimizer $\bm g_l$ is expected to satisfy 
\begin{align*}
    \left(
    \bm E_{\bm g_l,k_l}^{\mathrm{in}},
    \bm H_{\bm g_l,k_l}^{\mathrm{in}}
    \right)\big|_{\Omega}
    \approx
    \left(
    \bm E_{k_l},\bm H_{k_l}
    \right),
\end{align*}
where $(\bm E_{k_l},\bm H_{k_l})$ are the electromagnetic components of the EEITEP \eqref{eq_intro_EHuOmega} associated with $k_l$. Since $(\bm E_{k_l},\bm H_{k_l})$ obtained in the present work are localized near $\partial\Omega$, the reconstructed Maxwell-Herglotz wave functions are expected to inherit the same boundary-concentration behavior. This provides the basis for constructing the imaging indicator in Phase III. The rigorous justification and stability analysis of this approximation procedure for the system \eqref{eq_intro_EE_R3} and \eqref{eq_intro_SMRC} are left for future work.

\medskip 

\textit{Phase III: Imaging of the scatterer based on boundary-localized electromagnetic transmission eigenfunctions.}  

For the candidate transmission eigenvalues $\Lambda_L=\{k_1,k_2,\ldots,k_L\}$ and the associated Maxwell-Herglotz kernels $\bm{g}_{k_l}$, $l=1,2,\ldots,L$, motivated by the imaging strategy developed in the acoustic and Maxwell electromagnetic settings \cite{chow2021surface,he2024invisibility}, we consider a candidate imaging functional
\begin{align*}
    I_{\Lambda_L}^{\mathrm{E}}\left(\bm{z}\right)=-\ln \sum_{l=1}^L \left|\bm E_{\bm{g}_{k_l},k_l}^{\mathrm{in}}\left(\bm{z}\right)\right|.
\end{align*}
If the reconstructed Maxwell-Herglotz wave functions $\left(\bm E_{\bm{g}_{k_l},k_l}^{\mathrm{in}},\bm H_{\bm{g}_{k_l},k_l}^{\mathrm{in}}\right)$ approximate the boundary-localized transmission eigenfunctions, then they are expected to be concentrated near the boundary of the elastic scatterer $\Omega$. Consequently, $I_{\Lambda_L}^{\mathrm{E}}\left(\bm{z}\right)$ is expected to exhibit a contrast between the interior domain and the neighborhood of the boundary, which can be used to identify the boundary of the scatterer.

The quantitative lower bounds for the $L^{\infty}$-norms of the electromagnetic gradients derived in this paper further indicate that the boundary-localized electromagnetic components of the transmission eigenfunctions are sensitive to the boundary. However, these lower bounds only ensure large gradient magnitude at certain points and do not imply uniform pointwise enhancement along the entire boundary. Hence, they do not directly yield a gradient-based imaging functional for recovering $\partial\Omega$. Their precise role in enhancing the resolution of the above imaging scheme therefore remains to be investigated. Establishing the detection mechanism of transmission eigenvalues in Phase I, the Maxwell-Herglotz approximation property in Phase II, and the stability and numerical performance of the candidate indicator in Phase III will be pursued in future work.

\section*{Acknowledgments}
The work of H. Diao is supported by National Natural Science Foundation of China  (No. 12371422), the Fundamental Research Funds for the Central Universities, JLU. The work of H. Liu is supported by the Hong Kong RGC General Research Funds (projects 11311122, 11300821, and 11303125), the NSFC/RGC Joint Research Fund (project N\_CityU101/21), the France-Hong Kong ANR/RGC Joint Research Grant, A-CityU203/19.

\end{document}